\documentclass[psamsfonts]{amsart}

\usepackage{amssymb,amsfonts,amscd}
\usepackage[all,arc]{xy}
\usepackage{enumerate}
\usepackage{mathrsfs}

\newtheorem{thm}{Theorem}[section]
\newtheorem{cor}[thm]{Corollary}
\newtheorem{prop}[thm]{Proposition}
\newtheorem{lem}[thm]{Lemma}

\newtheorem*{thm1}{Theorem A}

\theoremstyle{definition}
\newtheorem{defn}[thm]{Definition}

\theoremstyle{remark}
\newtheorem{rem}[thm]{Remark}

\makeatletter
\let\c@equation\c@thm
\makeatother
\numberwithin{equation}{section}

\title[]{Semi-local simple connectedness\\of non-collapsing Ricci limit spaces}

\author[]{Jiayin Pan \& Guofang Wei}
\thanks{G.W. is partially supported by NSF DMS 1811558}
\thanks{J.P. is partially supported by AMS Simons travel grant}

\newcommand{\Addresses}{{ 
		\bigskip
		\footnotesize
		
		Jiayin Pan, \textsc{Department of Mathematics, University of California\ -\ Santa Barbara, CA, USA.}\par\nopagebreak
		\textit{E-mail address}: \texttt{jypan10@gmail.com}\\

        Guofang Wei, \textsc{Department of Mathematics, University of California\ -\ Santa Barbara, CA, USA.}\par\nopagebreak
        \textit{E-mail address}: \texttt{wei@math.ucsb.edu}

}}

\begin{document}
	
\begin{abstract}
	Let $X$ be a non-collapsing Ricci limit space and let $x\in X$. We show that for any $\epsilon>0$, there is $r>0$ such that every loop in $B_t(x)$ is contractible in $B_{(1+\epsilon)t}(x)$, where $t\in(0,r]$. In particular, $X$ is semi-locally simply connected.
\end{abstract}
	
\maketitle

\section*{Introduction}\label{section_intro}

Studying Gromov-Hausdorff limits of manifolds with uniform curvature lower bounds has been very active and has applications in many different directions. With sectional curvature lower bounds, the limit spaces are known as Alexandrov spaces, whose geometrical and topological structures have been extensively studied. 
In particular, these limit spaces are locally contractible \cite{Per1}.
For Ricci curvature lower bounds, Cheeger, Colding, and Naber have developed a rich theory on the regularity and geometric structure of the Ricci limit spaces. On the other hand, surprisingly little is known about the topology of these spaces. In fact, it could be so complicated that even a non-collapsing Ricci limit space may have locally infinite topological type \cite{Men}. About twenty years ago, Sormani and the second author \cite{SW1, SW2} gave the first topological restriction, showing that the universal cover of any Ricci limit space does exist, but were not able to show that the universal cover is simply connected. Recall that a connected and locally path-connected topological space has a simply connected universal cover if and only if space itself is semi-locally simply connected. See \cite[page 84]{Sp} for an example whose universal cover exists but is not simply connected.

As the main result of this paper, we show that any non-collapsing Ricci limit space is semi-locally simply connected.

\begin{thm1}\label{main}
	Let $(M_i,p_i)$ be a sequence of complete Riemannian $n$-manifolds converging to $(X,p)$ in the Gromov-Hausdorff topology with
	$$\mathrm{Ric}_{M_i}\ge -(n-1), \quad \mathrm{vol}(B_1(p_i))\ge v>0.$$
	Then $X$ is semi-locally simply connected.
\end{thm1}

Theorem A implies that any non-collapsed Ricci limit space has a simply connected universal cover. Our proof does not depend on the results in \cite{SW1,SW2}. 

In fact, we show that $X$ is essentially locally simply connected and have a local version. To state the result precisely, we use the notion of module of 1-contractibility (see \cite{Bor, Pet}). 

\begin{defn}\label{def_module}
	Let $X$ be a metric space. For $x\in X$ and $t\ge0$, we define $\rho(t,x)$, \textit{module of 1-contractibility} at $x$, as below:
	$$\rho(t,x)=\inf\{\infty,\rho\ge t| \text{ every loop in } B_t(x) \text{ is contractible in }B_\rho(x)\},$$
	where $B_r(x)$ is the open metric ball of radius $r$ centered at $x$.
\end{defn}

From the definition, it is clear that $X$ is semi-locally simply connected if for any $x\in  X$, there is $T>0$ such that $\rho(T,x)<\infty$; $X$ is locally simply connected if for any $x\in X$, there is $t_i\to 0$ such that $\rho(t_i,x)=t_i$. 

We state the local version of Theorem A with an estimate on $\rho(t,x)$.

\begin{thm}\label{main_local}
	Let $(M_i,p_i)$ be a sequence of Riemannian $n$-manifolds (not necessarily complete) converging to $(X,p)$ such that for all $i$,\\
	(1) $B_2(p_i)\cap\partial M_i=\emptyset$ and the closure of $B_2(p_i)$ is compact;\\
	(2) $\mathrm{Ric}\ge -(n-1)$ on $B_2(p_i)$, $\mathrm{vol}(B_1(p_i))\ge v>0.$\\
	Then 
	$$\lim\limits_{t\to 0}\dfrac{\rho(t,x)}{t}=1$$ 
	holds for any $x\in B_1(p)$.
\end{thm}

We explain some of the difficulties in studying semi-local simple connectedness. For Alexandrov spaces, Perelman showed there is a homeomorphism from the tangent cone at a point to a local neighborhood around this point \cite{Per1} (also see \cite{Ka}); together with the fact that tangent cones are metric cones \cite{BGP}, this leads to the local contractibility. However, for Ricci limit spaces, we no longer have such a connection between tangent cones and local topology. For non-collapsing Ricci limit spaces, Cheeger-Colding proved the important result that tangent cones at any point are all metric cones \cite{CC1}, but a neighborhood of a point could have infinite second Betti number \cite{Men}. Also, tangent cones at a point may not be unique and may not be homeomorphic \cite{CC2,CN}. Even at the fundamental group level, it is not clear how to connect the tangent cone to the neighborhood.  From the point of view of \cite{SW1}, since the universal cover always exists, it remains to rule out a non-contractible loop that can be homotopic to loops lying in arbitrarily small metric balls. Such a loop cannot be lifted to an open path in the universal cover, so one can not use deck transformations to study them. Also, such a loop may have infinite length. The last option is to use the sequence. In order to pass local simply connectedness information from the sequence to the limit space, one needs uniform control on the module of 1-contractibility for the sequence (see Theorem \ref{convergence_LGC}). However, a sequence of manifolds with the conditions in Theorem A may not have uniform modules of 1-contractibility. In fact, examples of Otsu show that the sequence may have shorter and shorter nontrivial loops \cite[page 262 Remark (2)]{Ot}. 

The most important step in our proof of Theorem \ref{main_local} is $\lim_{t\to 0}\rho(t,x)=0$. After proving that $\lim_{t\to 0}\rho(t,x)=0$ holds for all $x$, we can further improve the result to $\lim_{t\to 0}\rho(t,x)/t=1$, by using the structure of tangent cones, a modification of Sormani's uniform cut technique \cite{Sor}, and certain connection between the local fundamental group of the limit space and that of the sequence (see Section \ref{section_ratio}). 

We outline our approach to prove $\lim_{t\to 0}\rho(t,x)=0$ as follows. For a point $x$ in the limit space and a sequence of points $x_i$ on $M_i$, we classify the limit points into three types based on the module of contractibility at, or around, $x_i$. Roughly speaking, type I points are those that modules of $1$-contractibility are uniformly controlled in a uniform neighborhood around all $x_i$; type II points are those that modules of $1$-contractibility are not uniformly controlled at $x_i$; type III are the rest (see Definition \ref{def_types} and Lemma \ref{ind_bound}). When $x$ is a type I point, we can control $\rho(t,x)$ by constructing a sequence of uniformly convergent homotopies (see Lemma \ref{ind_bound} and Theorem \ref{convergence_LGC}); this is related to \cite{Bor}. 

The proof of $\lim_{t\to 0}\rho(t,x)=0$ is an induction argument on the local volume around $x$. As the base case, we start with points whose local volume is strictly larger than the half volume of the same size ball in the corresponding space form, or for simplicity, points with half volume lower bound. We show that $\rho(t,x_i)$ can be controlled by a linear function in this case,  where $x_i\in M_i$ converging to $x$; with this, we can deduce that $\lim_{t\to 0}\rho(t,x)=0$ holds (see Theorem \ref{large_vol_contra} and Proposition \ref{large_vol_type_1}). For the next induction step, we consider $x$ with quarter volume lower bound. If $x$ is of type I, then we are done. If $x$ is of type II, we can use small loops around $x_i$ to construct a sequence of covering spaces of certain local balls $B_\epsilon(x_i)$. On this sequence of covers, we can lift small loops as open paths. Moreover, these covers shall have half volume lower bound so that the base case in the induction can be applied. Anderson's results on small loops \cite{An1} and Cheeger-Colding's volume convergence for non-collapsing Ricci limit spaces \cite{Co,CC2} are essential in the above steps. The remaining type III case is the most technical situation. The method of type I points fails due to the lack of local control on $1$-contractibility. Also, there are no small loops to construct local covers. The key observation is that, based on results on type II and the definition of type III, for any point $x\in X$, at least one of $\{\rho(t,x_i)\}_i$ and $\rho(t,x)$ is controlled, where $x_i\in M_i$ converging to $x\in X$. This inspires us to construct each small piece of the desired homotopy from the data of the sequence $M_i$ or from that of the limit space $X$, in a delicate way (see more explanations in Section \ref{section_homotopy}). With the result on type III points (Theorem \ref{type_3_ok}), we can continue the induction argument and eventually finish the proof of $\lim_{t\to 0}\rho(t,x)=0$.

As a by-product from the study of points with half volume lower bound, we have the following. 
\begin{cor}\label{quant_anderson}
	Given $n$ and $L\in(1/2,1]$, there is a constant $C(n,L)$ such that the following holds.
	
	Let $M$ be a complete non-compact $n$-manifold of $\mathrm{Ric}\ge 0$. If $M$ has Euclidean volume growth of constant $L>1/2$, that is,
	$$\limsup\limits_{R\to\infty}\dfrac{\mathrm{vol}(B_R(p))}{\mathrm{vol}(B_R^n(0))}=L>\dfrac{1}{2}$$
	for some $p\in M$,
	then for any $x\in M$, any $r>0$, and any loop $c$ in $B_r(x)$, $c$ must be contractible in $B_{Cr}(x)$.
\end{cor}

Li and Anderson independently showed that if $M$ has Euclidean volume growth of constant $\ge L$, then $\pi_1(M)$ has order $\le 1/L$ \cite{Li,An2}; consequently, if $L>1/2$, then $M$ is simply connected. Hence Corollary \ref{quant_anderson} can be viewed as a quantitative description of the simple connectedness when $L>1/2$. We mention that Corollary \ref{quant_anderson} holds for non-collapsing Ricci limit spaces as well (see Theorem \ref{quant_limit}). The simple connectedness of open non-collapsing Ricci limit spaces is previously known only when $L$ is very close to $1$ \cite{Mu}.

With Theorem A, we can naturally generalize structure result of fundamental groups of manifolds with Ricci curvature and volume bounded below to that of non-collapsing Ricci limit spaces (see Section \ref{section_fund_group}).

We organize our paper as follows.

\tableofcontents

Acknowledgments: The authors would like to thank Professor Vitali Kapovitch for pointing out a minor error involving local covers (used in the proof of Theorem 3.5) in an earlier version of this paper. The authors would like to thank Jikang Wang for letting us know a minor error on estimating $\rho(t,x)$ from a nearby point.

\section{Convergence of spaces with controlled 1-contractibility}\label{section_conv}

We study Gromov-Hausdorff convergence of spaces with uniformly controlled module of $1$-contractibility. This is related to \cite{Bor}, where Hausdorff convergence and the contractibility of subsets are considered.

By Definition \ref{def_module}, $\rho(t,x)=L<\infty$ means that for any $\sigma>0$ and any loop $c$ contained in $B_t(x)$, there is a homotopy between $c$ and a trivial loop with the image of $H$ contained in $B_{L+\sigma}(x)$. Throughout the text, we always use a term involving $\sigma$ for this situation (for example, a term like $\sigma2^{-i}$).

In general, $\rho(t,x)$ may not be a continuous function. On the other hand, we can always bound $\rho(t,x)$ by a so called indicatrix, which is continuous \cite{Bor}.

\begin{defn}\cite{Bor}
	Let $T\in (0,1)$ and $\lambda:[0,T)\to [0,1]$ be a function. We say that $\lambda$ is an indicatrix, if $\lambda$ is continuous, non-decreasing, and concave with $\lambda(0)=0$.
\end{defn}

We always assume that $T\in(0,1)$ in this paper unless otherwise noted.

\begin{lem}\label{ind_bound}
	Let $\{\rho_\alpha(t)\}_{\alpha\in A}$ be a family of non-decreasing functions on $[0,T)$ with
	$$0=\rho_\alpha(0)=\lim\limits_{t\to 0} \rho_\alpha(t), \quad 0\le\rho_\alpha(t)\le 1 \text{ for all } t\in[0,T)$$
	for every $\alpha\in A$. Then the following two statements are equivalent:\\
	(1) There is an indicatrix $\lambda(t)$ on $[0,T)$ such that $\rho_\alpha(t)\le \lambda(t)$ for all $t\in [0,T)$ and all $\alpha\in A$.\\
	(2) The family $\{\rho_\alpha(t)\}_{\alpha\in A}$ is equally continuous at $0$.
\end{lem}

\begin{proof}
	It is clear that (1) implies (2).
	
	Conversely, suppose that (2) holds. Consider $g(t)=\sup_{\alpha\in A} \rho_\alpha(t)$, which satisfies $$\lim\limits_{t\to 0}g(t)=g(0)=0$$ by assumption. By \cite[Section 7]{Bor}, there is a concave and non-decreasing function $\lambda(t)$ such that 
	$$g(t)\le \lambda(t)$$
	on $[0,T)$. It is clear that $\lambda(t)$ is continuous on $(0,T)$ because it is concave. Also, according to (35) in \cite{Bor}, $\lim_{t\to 0}\lambda(t)=0$ holds.
\end{proof}

\begin{defn}
	Let $\epsilon>0$. We say that two loops $c,c':[0,1]\to X$ are $\epsilon$-close, if $d(c(t),c'(t))\le \epsilon$ for all $t\in[0,1]$.
\end{defn}

The lemma below illustrates a relation between Gromov-Hausdorff closeness and homotopies. The method is similar to \cite{Tu,Pet}. Because we need this construction and its related estimates later, we include the proof for readers' convenience.

\begin{lem}\label{retract_homotopy}
	Given $T\in(0,1)$, there is $\epsilon_0=T/20$ such that the following holds:
	
	Let $(X,x)$ and $(Y,y)$ be two length metric spaces with the conditions below:\\
	(1) the closure of $B_2(p)$ is compact, where $p=x$ or $y$;\\
	(2) $d_{GH}((X,x),(Y,y))\le \epsilon\le\epsilon_0$;\\
	(3) for any $q\in B_{1}(x)$ and any loop $\gamma$ contained in $B_T(q)$, $\gamma$ is contractible in $B_1(q)$.\\
	Then\\
	(i) For any loop $c$ in $B_1(y)\subset Y$, there is a loop $c'$ in $X$ that is $5\epsilon$-close to $c$. Moreover, if $c''$ is another loop in $X$ that is $5
	\epsilon_0$-close to $c$, then $c''$ is free homotopic to $c'$ in $B_2(x)$.\\
	(ii) Let $c_i$ $(i=1,2)$ be a loop in $B_1(y)\subset Y$ and $c'_i$ be a loop in $X$ that is $5\epsilon$-close to $c_i$. If $c_1$ and $c_2$ are free homotopic in $B_1(y)$ via a homotopy
	$$H: S^1\times [0,1]\to B_1(y),$$
	then there is a continuous map $H':S^1\times[0,1]\to X$ such that $H'$ is a free homotopy in $B_2(x)$ between $c'_1$ and $c'_2$.
\end{lem}

\begin{proof}
	(i) Let $c:[0,1]\to B_1(y)$ be a loop. Because $c$ is uniform continuous, we can choose a large integer $N$ such that
	$$\mathrm{diam}(c|_{[i/N,(i+1)/N]})\le\epsilon$$
	for all $i=0,1,...,N-1$. For each $i$, we choose $q_i\in B_1(x)$ such that $d(c(i/N),q_i)\le\epsilon$. Next we connect $q_i$ to $q_{i+1}$ by a minimal geodesic for each $i$ and close it up as a loop by connecting $q_{N-1}$ to $q_0$ by a minimal geodesic. In this way, we result in a loop $c'$, as a broken geodesic, in $X$. Note that
	$$d(q_i,q_{i+1})\le d(q_i,c(i/N))+d(c(i/N),c((i+1)/N))+d(c((i+1)/N),q_{i+1})\le 3\epsilon.$$
	Re-parameterize $c'$ if necessary, we can assume that $c'(i/N)=q_i$ and
	$$d(q_i,c'(i/N+t))=tNd(q_i,q_{i+1})$$
	for $t\in[0,1/N]$. Then
	$$d(c(t),c'(t))\le d(c(t),c(i/N))+d(c(i/N),q_i)+d(q_i,c'(t))\le 5\epsilon$$
	where $i$ is chosen such that $t\in [i/N, (i+1)/N]$.
	Thus $c'$ is $5\epsilon$-close to $c$.
	
	Let $c''$ be another loop that is $5\epsilon_0$-close to $c$. Then
	$$d(c'(t),c''(t))\le 5\epsilon+5\epsilon_0\le 10\epsilon_0$$
	for all $t\in[0,1]$. Choose a large integer $L$ such that
	$$\mathrm{diam}(c'|_{[i/L,(i+1)/L]})\le \epsilon_0, \quad \mathrm{diam}(c''|_{[i/L,(i+1)/L]})\le \epsilon_0$$
	for all $i=0,1,...,L-1$. Let $l_i$ be a minimal geodesic from $c'(i/L)$ to $c''(i/L)$. Since the small loop $c'|_{[i/L,(i+1)/L]}\cdot l_{i+1}\cdot (c''|_{[i/L,(i+1)/L]})^{-1}\cdot l_i^{-1}$ is contained in $B_{6\epsilon_0}(c'(i/L))\subset B_T(c'(i/L))$, it is contractible in $B_2(x)$. We conclude that $c'$ is free homotopic to $c''$ in $B_2(x)$.
	
	(ii) We first define $H'$ on the boundary $S^1\times\{0,1\}$ so that
	$$H'(t,0)=c'_1(t),\quad H'(t,1)=c'_2(t).$$
	This implies for all $t\in[0,1]$ and $s=0$ or $1$,
	$$d(H(t,s),H'(t,s))\le 5\epsilon.$$
	We choose a finite triangular decomposition $\Sigma$ of $S^1\times [0,1]$ such that $\mathrm{diam}(H(\Delta))\le \epsilon$ for any triangle $\Delta$ of $\Sigma$. Let $K^0$ be the set of all vertices of $\Sigma$ and let $K^1$ be the $1$-skeleton of $\Sigma$. If $v\in K^0$ is in the boundary $S^1\times\{0,1\}$, we have already defined $H'(v)$. If $v\in K^0$ is not in the boundary, we define $H'(v)$ to be a point in $X$ with
	$$d(H(v),H'(v))\le\epsilon.$$
	If two vertices $v$ and $w$ of $K^0$ is connected by an edge that is not part of the boundary, then we connect $H'(v)$ and $H'(w)$ by a minimal geodesic in $X$. From this, we obtain a continuous map, which we still call $H'$, $H':K^1\to X$. Let $\Delta$ be a triangle of $\Sigma$ and let $\partial\Delta$ be its boundary. By our construction, it is direct to check that if a triangle $\Sigma$ does not have any boundary point as its vertex, then\footnote{We always use extrinsic distance to measure diameter of a subset in this paper.}
	$$\mathrm{diam}(H'(\partial\Delta))\le 5\epsilon,$$
	where $\partial\Delta$ is the $1$-skeleton of $\Delta$;
	if $\Sigma$ has one or more boundary points as vertices, then
	$$\mathrm{diam}(H'(\partial\Delta))\le 15\epsilon.$$
	In particular, $H'(\partial\Delta)$ is contained in $B_{18\epsilon}(H'(v))$ with $15\epsilon<T$, where $v$ is a vertex of $\Delta$. By assumption $H'(\partial\Delta)$, as a loop, is contractible in $B_1(v')$, thus we can extend the domain of the map $H'$ from $\partial\Delta$ to $\Delta$. Since the extension can be achieved over each $\Delta$, we result in the desired free homotopy.
\end{proof}

\begin{rem}
	In Lemma \ref{retract_homotopy}(i), the existence of $c'$ does not require condition (3).
\end{rem}

If the module of 1-contractibility of any point $q\in B_1(x)$ is bounded by an indicatrix $\lambda(t)$, then we can control the homotopy $H'$ constructed in Lemma \ref{retract_homotopy} so that $H'$ is close to $H$.

\begin{lem}\label{retract_homotopy_dist}
	Given an indicatrix $\lambda$ on $[0,T)$, there is $\epsilon_0=T/20$ and a function $\phi(\epsilon)$ with $\lim\limits_{\epsilon\to 0}\phi(\epsilon)=0$ such that the following holds:
	
	Let $(X,x)$ and $(Y,y)$ be two length metric spaces with the conditions (1) and (2) in Lemma \ref{retract_homotopy} and (3') below:\\
	(3') $\rho(t,q)\le\lambda(t)$ for all $t\in[0,T)$ and all $q\in B_1(x)$.\\
	Then for any $\sigma>0$, we can construct a free homotopy $H'$ as in Lemma \ref{retract_homotopy}(ii) satisfying $d(H(z),H'(z))\le\phi(\epsilon)+\sigma$ for all $z\in S^1\times[0,1]$.
\end{lem}

\begin{proof}
	We continue to use the notations in the proof of Lemma \ref{retract_homotopy}(ii). Note that by assumption (3'), for any $\sigma>0$, there is a homotopy from $H'(\partial\Delta)$ to a trivial loop so that the image of the homotopy is contained in $B_{\lambda(18\epsilon)+\sigma}(H'(v))$.
	Thus for all $z\in \Delta$,
	\begin{align*} 
		d(H(z),H'(z))&\le d(H(z),H(v))+d(H(v),H'(v))+d(H'(v),H'(z))\\
		&\le \epsilon+5\epsilon+\lambda(15\epsilon)+\sigma\\
		&= \phi(\epsilon)+\sigma,
	\end{align*}
	where $\phi(\epsilon)=\lambda(15\epsilon)+6\epsilon$.
\end{proof}

With Lemma \ref{retract_homotopy_dist}, we show that locally controlled module of 1-contractibility is preserved under Gromov-Hausdorff convergence.

\begin{thm}\label{convergence_LGC}
	Let $(X_i,x_i)$ be a sequence of length metric spaces with the conditions below:\\
	(1) the closure of $B_2(x_i)$ is compact;\\
	(2) there exists an indicatrix $\lambda$ on $[0,T)$ such that for all $i$ and all $q\in B_2(x_i)$, $\rho(t,q)\le\lambda(t)<1/2$ holds on $[0,T)$;\\
	(3) $(X_i,x_i)\overset{GH}\longrightarrow (Y,y)$.\\
	Then $\rho(t,q)\le\lambda(t)$ for all $t\in[0,T)$ and all $q\in B_{3/2}(y)$.
\end{thm}

We state the conclusion of above theorem for $B_{3/2}(y)$ instead of $B_1(y)$, because we need some extra room for later use in Section \ref{section_class}. 

The proof below is inspired by \cite{Bor}, where Hausdorff convergence and the contractibility of subsets, instead of loops, are considered. Besides these differences, compared with \cite[Section 15]{Bor}, our statement is localized; also the proof is much simplified and streamlined with the length metric space condition. Later in Section \ref{section_homotopy}, we will present a different proof of Theorem \ref{convergence_LGC}, whose strategy is more in line with our main construction.

\begin{proof}[Proof of Theorem \ref{convergence_LGC}]
	Let $\sigma>0$. Let $\epsilon_0$ and $\phi(\epsilon)$ as in Lemma \ref{retract_homotopy_dist}. We choose a decreasing sequence $\epsilon_i\to 0$ such that $\epsilon_1\le \epsilon_0/20$, $10\epsilon_i\le 2^{-i}$, and $\phi(2\epsilon_i)\le 2^{-i}$ for all $i$. Passing to a subsequence, we can assume that $$d_{GH}(B_2(x_i),B_2(y))\le\epsilon_i\to 0$$
	for all $i$.
	
	Fix $t\in (0,T)$. We fix an $I$ large so that 
	$$t+6\epsilon_I<T,\quad \lambda(t+6\epsilon_I)+(3\sigma+4)2^{-I}<1/2.$$ Let $q\in B_{3/2}(y)$ and let $c$ be a loop in $B_t(q)$. We will construct a homotopy between $c$ and a trivial loop. For each $j\ge I$, pick $q_j\in B_{3/2}(x_j)$ with $d(q_j,q)\le \epsilon_j$. 
	Let $c_j$ be a loop in $B_{3/2}(x_j)$ that is $5\epsilon_j$-close to $c$. It is clear that
	$$d(c_j(t),c_{j+1}(t))\le d(c_j(t),c(t))+d(c(t),c_{j+1}(t))\le 5\epsilon_j+5\epsilon_{j+1}\le10\epsilon_j\le 2^{-j}.$$
	In particular, $c_j$ converges uniformly to $c$. Since the image of $c_I$ satisfies
	$$\mathrm{im} (c_I)\subset B_{t+6\epsilon_I}(q_I)\subset B_T(q_I),$$
	there is a homotopy $$H_I:D\to B_{\lambda(t+6\epsilon_I)+\sigma\cdot 2^{-I}}(q_I)$$
    between $c_I$ and a trivial loop, where $D$ is the closed unit disk.
	By Lemma \ref{retract_homotopy_dist}, we can construct a homotopy between $c_{I+1}$ and a trivial loop
	$$H_{I,I+1}: D \to B_2(x_{I+1})\subset X_{I+1}$$ such that $$d(H_I(z),H_{I,I+1}(z))\le \phi(2\epsilon_I)+\sigma\cdot 2^{-I}\le (1+\sigma)2^{-I}$$
	for all $z\in D$.
    Also,
    \begin{align*}
    	d(H_{I,I+1}(z),q_{I+1})&\le d(H_{I,I+1}(z),H_I(z))+d(H_I(z),q_I)+d(q_I,q_{I+1})\\
    	&\le (1+\sigma)2^{-I}+\lambda(t+6\epsilon_I)+\sigma2^{-I}+2\epsilon_I\\
    	&\le\lambda(t+6\epsilon_I)+\sigma2^{-I}+(2+\sigma)2^{-I}.
    \end{align*}
    Thus
    $$\mathrm{im}(H_{I,I+1})\subset B_{\lambda(t+6\epsilon_I)+\sigma2^{-I}+(2+\sigma)2^{-I}}(q_{I+1}).$$
    
	In general, for $j\ge I+1$, suppose that we have constructed a homotopy $H_{I,j}$ between $c_j$ to a trivial loop with
	$$\mathrm{im} (H_{I,j})\subset B_{\lambda(t+6\epsilon_I)+\delta_{I,j}}(q_{I+1}),\quad d(H_{I,j}(z),H_{I,j-1}(z))\le (1+\sigma)2^{-(j-1)},$$
	where $$\delta_{I,j}=\sigma2^{-I}+\sum_{k=I}^j (2+\sigma)2^{-k}\le (3\sigma+4)\cdot 2^{-I}.$$
	We apply Lemma \ref{retract_homotopy_dist} again to obtain $$H_{I,j+1}:D\to B_2(x_{j+1})\subset X_{j+1}$$
	with $d(H_{I,j+1}(z),H_{I,j}(z))\le (1+\sigma)\cdot 2^{-j}$.
	It is direct to check that
	$$\mathrm{im} (H_{I,j+1})\subset B_{\lambda(t+6\epsilon_I)+\delta_{I,j+1}}(q_{I+1}).$$
	This process gives us a sequence of homotopies
	$\{H_{I,j}\}_{j=I}^\infty$.
	Since $$d(H_{I,j}(z),H_{I,j+1}(z))\le 2^{-j}(1+\sigma)$$ for all $z\in D$ and all $j\ge I$, we conclude that as $j\to\infty$, $H_{I,j}$ converges uniformly to a continuous map $H_{I,\infty}:D\to Y$, which is a homotopy between $c$ and a trivial loop. Moreover, we have
	$$\mathrm{im} (H_{I,\infty})\subset B_{\lambda(t+6\epsilon_I)+(3\sigma+4)\cdot 2^{-I}}(q).$$ 
	Noting that $I$ can be arbitrarily large and loop $c$ is arbitrary in $B_t(q)$, we see $\rho(t,q)\le \lambda(t)$.
\end{proof}

\begin{rem}\label{rem_move_homotopies}
	From the proof of Lemma \ref{retract_homotopy} and Theorem \ref{convergence_LGC}, we see two methods to move homotopies. The first method works when the target space has local contractibility (Lemma \ref{retract_homotopy}): we dissemble the homotopy into small pieces, then map the $1$-skeleton to the target space then obtain the homotopy via extensions. The second method works when there is a sequence of spaces with uniformly controlled local contractibility converging to the target space (Theorem \ref{convergence_LGC}): we transfer the homotopy along the sequence and pass it to the target space by uniform convergence.
\end{rem}

\section{Half volume lower bound and 1-contractibility}\label{section_half_volume}

In this section, we show that if the local volume is strictly larger than the half volume of a same size ball in the corresponding space form, then the module of $1$-contractibility is controlled by a linear function. As one of applications, this control implies that $\lim_{t\to 0}\rho(t,x)=0$ holds for $x$ in the limit space if $x$ has a local half volume lower bound (see Proposition \ref{large_vol_type_1}).

For $\kappa\in \mathbb{R}$ and $r>0$, we denote the $r$-ball in the $n$-dimensional space form of curvature $\kappa$ as $B^n_r(\kappa)$.

\begin{thm}\label{large_vol_contra}
	Given $n\ge 2$, $\kappa\ge 0$, and $\omega>1/2$, there exist positive constants $\epsilon(n,\kappa,\omega)$ and $C(n,\kappa,\omega)$ such that the following holds.
	
	Let $(M,p)$ be a Riemannian $n$-manifold satisfying\\
	(1) $B_2(p)\cap \partial M=\emptyset$ and the closure of $B_2(p)$ is compact,\\
	(2) $\mathrm{Ric}\ge -(n-1)\kappa$ on $B_2(p)$, $\mathrm{vol}(B_1(p))\ge \omega\cdot \mathrm{vol}(B_1^n(-\kappa)).$\\
	Then every loop in $B_r(p)$ is contractible in $B_{Cr}(p)$, where $r\in [0,\epsilon)$.
\end{thm}

Because of Otsu's example \cite{Ot}, if the half volume lower bound is replaced by a positive lower bound, then the conclusion of Theorem \ref{large_vol_contra} would fail.

One may compare Theorem \ref{large_vol_contra} with Grove-Petersen's result on sectional curvature:

\begin{thm}\cite{GP}
	Given $n,\kappa,v>0$, there exist positive constants $\epsilon(n,\kappa,v)$ and $C(n,\kappa,v)$ such that the following holds.
	
	Let $(M,p)$ be a complete $n$-manifold of
	$$\mathrm{sec}_M\ge -\kappa,\quad \mathrm{vol}(B_1(p))\ge v.$$
	Then $B_r(p)$ is contractible in $B_{Cr}(p)$, where $r\in [0,\epsilon)$.
\end{thm}

\begin{lem}\label{large_vol_contra_lem}
	Given $n, \kappa\ge 0$, $R\in(0,1)$, and $\omega>1/2$, there exists a positive constant $\epsilon(n,\kappa,\omega,R)$ such that the following holds.\\
	(1) Let $(M,p)$ be a Riemannian $n$-manifold with the conditions in Theorem \ref{large_vol_contra}.
	Then every loop in $B_\epsilon(p)$ is contractible in $B_{R}(p)$.\\
	(2) $\limsup\limits_{R\to 0^+} \dfrac{R}{\epsilon(n,\kappa,\omega,R)}<\infty$.
\end{lem}

\begin{proof}
	(1) To find $\epsilon(n,\kappa,\omega,R)$, we argue by contradiction. Then we would have a sequence $(M_i,p_i)$ of
	$$\mathrm{Ric}_{M_i}\ge -(n-1)\kappa,\quad \mathrm{vol}(B_1(p_i))\ge \omega\cdot \mathrm{vol}(B_1^n(-\kappa));$$
	Moreover, each $B_{\epsilon_i}(p_i)$ contains a loop $\gamma_i$ that is not contractible in $B_R(p_i)$. Note that by relative volume comparison, we have
	$$\mathrm{vol}(B_R(p))\ge \dfrac{\mathrm{vol}(B_1(p))}{\mathrm{vol}(B_1^n(-\kappa))}\cdot\mathrm{vol}(B_R^n(-\kappa))\ge \omega\cdot\mathrm{vol}(B_R^n(-\kappa))$$
	for all $R\in (0,1)$.
	Passing to a subsequence, we assume $(M_i,p_i)\overset{GH}\longrightarrow(X,p)$. Without lose of generality, we also assume that each $\gamma_i$ has length less than $4\epsilon_i$.
	
	Let $(U_i,y_i)$ be the universal covering space of $(B_R(p_i),p_i)$ with covering group $H_i=\pi_1(B_R(p_i),p_i)$. Let $\Gamma_i$ be the subgroup generated by $\gamma_i$. By \cite{An1}, each $\Gamma_i$ is a finite group with order $\le N(n,\kappa,\omega,R)$ (The statement in \cite{An1} is about compact manifolds; its proof extends clearly to local balls).
	Since $d(\gamma_iy_i,y_i)\le 4\epsilon_i$ and $\gamma_i$ has order $\le N$ for all $i$, it follows that $$\mathrm{diam}(\Gamma_i\cdot y_i)=N\cdot 4\epsilon_i\to 0.$$
	Then
	\begin{align*}
		2\cdot\omega\cdot\mathrm{vol}(B_{R/2}^n(-\kappa))&\le \# \Gamma_i\cdot \mathrm{vol}(B_{R/2}(p_i))\\
		&= \# \Gamma_i\cdot \mathrm{vol}(B_{R/2}(y_i)\cap F_i)\\
		&= \sum_{\gamma\in \Gamma_i} \mathrm{vol}(\gamma\cdot(B_{R/2}(y_i)\cap F_i))\\
		&\le \mathrm{vol}(B_{R/2+4N\epsilon_i}(y_i))\\
		&\le \mathrm{vol}(B_{R/2+4N\epsilon_i}^n(-\kappa))\\
		&\to \mathrm{vol}(B_{R/2}^n(-\kappa)),
	\end{align*}
	where $F_i$ is the Dirichlet domain centered at $y_i$. This clearly leads to a contradiction since $2\omega>1$.
	We complete the proof of (1).
	
	(2) Suppose that there is $R_i\to 0$ with $R_i/\epsilon_i\to \infty$, where $\epsilon_i=\epsilon(n,\kappa,\omega,R_i)$. This means that we can find a sequence of $n$-manifolds $(M_i,p_i)$ with the conditions below:\\
	(i) $\mathrm{Ric}_{M_i}\ge -(n-1)\kappa,\  \mathrm{vol}(B_1(p_i))\ge \omega\cdot \mathrm{vol}(B_1^n(-\kappa));$\\
	(ii) for each $i$ there is a loop in $B_{2\epsilon_i}(p_i)$ but not contractible in $B_{R_i}(p_i)$.\\
	We rescale the sequence $(B_{R_i}(p_i),p_i)$ by $R_i^{-1}$. This rescaled sequence
	$$(R_i^{-1}B_{R_i}(p_i),p_i)$$
	satisfies\\
	(i') $\mathrm{Ric}\ge -(n-1)R_i^{-2}\kappa\to 0$, and the unit ball centered at $p_i$ has volume
	$$R_i^{-n}\mathrm{vol}(B_{R_i}(p_i))\ge R_i^{-n}\cdot\omega\cdot\mathrm{vol}(B_{R_i}^n(-\kappa))= \omega\cdot\mathrm{vol}(B_1^n(-R_i^{-2}\kappa))\to\omega\cdot\mathrm{vol}(B_1^n(0)).$$
	(ii') for each $i$ there is a loop $\gamma_i$ contained in the ball of radius $2\epsilon_iR_i^{-1}(\to 0)$ and $\gamma_i$ is not contractible in the concentric unit ball.\\
	We apply the argument used in (1) once again and result in a contradiction.
\end{proof}

\begin{proof}[Proof of Theorem \ref{large_vol_contra}]
	Let 
	$$L(n,\kappa,\omega):=\limsup\limits_{R\to 0^+} \dfrac{R}{\epsilon(n,\kappa,\omega,R)}<\infty$$ as in Lemma \ref{large_vol_contra_lem}(2). Choose $R_0>0$ such that $R/\epsilon(n,\kappa,\omega,R)\le 2L$ holds for all $R\le R_0$. Then for any loop $\gamma$ contained in $B_r(p)$, where $r=R/(2L)\le R_0/(2L)$, since $r\le\epsilon(n,\kappa,\omega,R)$, $\gamma$ must be contractible in $B_R(p)=B_{2Lr}(p)$.
\end{proof}

Corollary \ref{quant_anderson} follows directly from Theorem \ref{large_vol_contra}.

\begin{proof}[Proof of Corollary \ref{quant_anderson}]
	By relative volume comparison,
	$$\dfrac{\mathrm{vol}(B_R(x))}{\mathrm{vol}(B_R^n(0))}\ge L$$
	holds for all $x\in M$ and all $R>0$. Together with Theorem \ref{large_vol_contra}, the result follows immediately.
\end{proof}

\begin{cor}
	Given $n,\rho>0$ and $\omega\in (1/2,1]$, there is $\epsilon(n,\rho,\omega)>0$ such that the following holds:
	
	Let $M_i$ ($i=1,2$) be a compact Riemannian $n$-manifold with
	$$\mathrm{Ric}_{M_i}\ge -(n-1),\quad \mathrm{vol}(B_\rho(x))\ge \omega \cdot \mathrm{vol}(B_\rho^n(-1))$$
	for all $x\in M_i$.
	If
	$$d_{GH}(M_1,M_2)\le\epsilon,$$
	then $\pi_1(M_1)$ and $\pi_1(M_2)$ are isomorphic.
\end{cor}

\begin{proof}
	By Theorem \ref{large_vol_contra}, there are $T,C>0$, depending on $n$, $\rho$, and $\omega$, such that $\rho(t,x)\le C\cdot t$ for all  $t\in [0,T)$ and all $x\in M_i$ ($i=1,2$). The result follows immediately from \cite[Corollary 2.3]{SW1}.
\end{proof}

Now we use Theorem \ref{large_vol_contra} to study non-collapsing Ricci limit space. Let $X$ be a space with the assumptions in Theorem \ref{main_local}. By \cite{CC2}, $B_2(p)\subset X$ has Hausdorff dimension $n$.

\begin{defn}
	Let $x\in B_1(p)\subset X$, we define
	$$\omega(x)=\lim\limits_{r\to 0} \dfrac{\mathcal{H}^n(B_r(x))}{\mathrm{vol}(B^n_r(0))},$$
	where $\mathcal{H}^n$ is the $n$-dimensional Hausdorff measure in $X$.
\end{defn}

By relative volume comparison, this limit always exists and belongs to $(0,1]$. In fact, it is clear that $\omega(x)$ has a uniform positive lower bound for all $x\in B_1(p)$. 

\begin{prop}\label{large_vol_type_1}
	Let $x\in B_1(p)$ with $\omega(x)>1/2$. Then $\lim_{t\to 0}\rho(t,x)=0$.
\end{prop}

\begin{proof}
	For a fixed small $\epsilon>0$ with
	$$\omega':=\dfrac{\omega(x)-\epsilon}{1+\epsilon}>\dfrac{1}{2},$$
	there are $s_0,d_0>0$ such that
	$$\left|\dfrac{\mathcal{H}^n(B_{s_0}(x))}{\mathrm{vol}(B^n_{s_0}(-1))}-\omega(x)\right|\le\epsilon/2,\quad \dfrac{\mathrm{vol}(B_{s_0+d_0}^n(-1))}{\mathrm{vol}(B_{s_0}^n(-1))}<1+\epsilon.$$
	By volume convergence \cite{Co,CC2}, there is $N>0$ such that for all $i\ge N$, we have
	$$\left|\dfrac{\mathrm{vol}(B_{s_0}(x_i))}{\mathrm{vol}(B^n_{s_0}(-1))}-\omega(x)\right|\le\epsilon.$$
	Let $y\in B_1(x_i)$ with $d=d(x_i,y)\le d_0$. When $t<s_0+d$ and $i\ge N$, we have
	\begin{align*}
		\omega(x)-\epsilon&\le\dfrac{\mathrm{vol}(B_{s_0}(x_i))}{\mathrm{vol}(B^n_{s_0}(-1))}\\
		&= \dfrac{\mathrm{vol}(B_{s_0}(x_i))}{\mathrm{vol}(B_{s_0+d}^n(-1))}\cdot \dfrac{\mathrm{vol}(B_{s_0+d}^n(-1))}{\mathrm{vol}(B_{s_0}^n(-1))}\\
		&\le \dfrac{\mathrm{vol}(B_{s_0+d}(y))}{\mathrm{vol}(B_{s_0+d}^n(-1))}\cdot (1+\epsilon)\\
		&\le \dfrac{\mathrm{vol}(B_{t}(y))}{\mathrm{vol}(B_{t}^n(-1))}\cdot(1+\epsilon)
	\end{align*}
	Put $t=s_0$, thus for all $i\ge N$ and $y\in B_{d_0}(x_i)$,
	$$\mathrm{vol}(B_{s_0}(y))\ge \omega'\cdot\mathrm{vol}(B_{s_0}^n(-1))$$
	By Theorem \ref{large_vol_contra}, there are $T,C>0$, depending on $n,s_0$, and $\omega'$, such that for any $i\ge N$ and any $y\in B_{d_0}(x_i)$, $\rho(t,y)\le C t$ holds for $t\in [0,T)$. By Theorem \ref{convergence_LGC}, in the limit we have $\rho(t,x)\le C t$ for $t\in [0,T)$.
\end{proof}

\begin{rem}\label{rem_omega_lower}
	By a similar volume estimate presented above, it is clear that if $\omega(x)>\omega_0$, then there is $r>0$ such that $\omega(y)>\omega_0$ for all $y\in B_r(x)$.
\end{rem}

\section{Classification of points by $1$-contractibility}\label{section_class}

We classify points in the limit space by the module of $1$-contractibility from the sequence.

\begin{defn}\label{def_types}
	Let $(X_i,p_i)$ be a sequence of length metric spaces converging to $(X,p)$. Let $x\in X$.
	
	We say that $x$ is of \textit{type I}, if there is $r>0$ such that the family of functions $\{\rho(q,t)|q\in B_r(x_i),i\in \mathbb{N}\}$ is equally continuous at $t=0$, where $x_i\in X_i$ converging to $x$.
	
	We say that $x$ is of \textit{type II}, if there is $x_i\in X_i$ converging to $x$ so that $\{\rho(x_i,t)\}_i$ is not equally continuous at $t=0$.
	
	We say that $x$ is of \textit{type III}, if it is not of type I nor type II.
\end{defn}

The following lemma assures that Definition \ref{def_types} is well-defined because it does not depend on the choice of $x_i$.

\begin{lem}\label{type_well_defined}
	Let $x_i,x'_i\in X_i$ and $x\in X$ with
	$$x_i\to x, \quad x'_i\to x$$
	as $i\to \infty$. Then $\{\rho(t,x_i)\}_i$ is equally continuous at $t=0$ if and only if $\{\rho(t,x'_i)\}$ is equally continuous at $t=0$. 
\end{lem}

\begin{proof}
	Suppose that the family $\{\rho(t,x_i)\}_i$ is not equally continuous at $t=0$, then there are $\epsilon>0$, $t_j\to 0$, and a subsequence $x_{i(j)}$ such that
	$$\rho(t_j,x_{i(j)})\ge \epsilon$$
	for all $j$. Let $d_i=d(x_i,x'_i)\to 0$. We have
	$$\epsilon \le \rho(t_j,x_{i(j)})\le \rho(t_j+d_{i(j)},x'_{i(j)})+d_{i(j)}.$$
	Since $t_j+d_{i(j)}\to 0$ as $j\to\infty$, we see that the family $\{\rho(t,x'_i)\}_i$ is not equally continuous at $t=0$ as well. This completes the proof.
\end{proof}

Note that by Definition \ref{def_types}, the set of type I points is open in $X$. Due to Otsu's example, type II points in general may exist under the assumption of Theorem A. For a type III point $x$ and a sequence $x_i\in X_i$ converging to $x$, by definition the family $\{\rho(t,x_i)|i\in\mathbb{N}\}$ is equally continuous at $t=0$. Since for any $r>0$, the family $\{\rho(q,t)|q\in B_r(x_i),i\in\mathbb{N}\}$ is not equally continuous at $t=0$, this implies that the closure of $B_r(x)$ must contain a point of type II. Hence any type III point must be a limit point of the set of type II points.

\begin{cor}\label{type_1_ok}
	Let $(X_i,p_i)$ be a sequence of length metric spaces with the conditions below:\\
	(1) the closure of $B_2(p_i)$ is compact;\\
	(2) $(X_i,p_i)\overset{GH}\longrightarrow (Y,p)$.\\
	If $x\in B_{3/2}(p)$ is of type I, then $\lim\limits_{t\to 0} \rho(t,x)=0$.
\end{cor}

\begin{proof}
	Since $x$ is of type I, by Lemma \ref{ind_bound} there is an indicatrix $\lambda(t)$ on $[0,T)$ such that for all $i$ and all $y\in B_r(x_i)$, $\rho(t,y)\le \lambda(t)$ holds on $[0,T)$. By Theorem \ref{convergence_LGC}, $x$ satisfies $\rho(t,x)\le \lambda(t)$ on $[0,T)$ and the result follows.
\end{proof}

Also, for $X$ with the conditions in Theorem \ref{main_local}, by the proof of Proposition \ref{large_vol_type_1}, any point $x\in B_{3/2}(p)$ with $\omega(x)>1/2$ is of type I.

We will prove the theorem below on type III points in Section \ref{section_homotopy}.

\begin{thm}\label{type_3_ok}
	Let $(X_i,p_i)$ be a sequence of length metric spaces with the conditions below:\\
	(1) the closure of $B_2(p_i)$ is compact;\\
	(2) $(X_i,p_i)\overset{GH}\longrightarrow (X,p)$.\\
	Suppose that $\lim_{t\to 0}\rho(t,x)=0$ holds for all points $x$ of type II in $B_{3/2}(p)$, then it holds for all points of type III in $B_1(p)$. Consequently, $\lim_{t\to 0}\rho(t,x)=0$ holds for all $x\in B_1(p)$.	
\end{thm} 

Assuming that Theorem \ref{type_3_ok} is true, then we can prove $\lim_{t\to 0} \rho(t,x)=0$ holds for all $x\in B_1(p)$ in the context of Theorem \ref{main_local} (see Theorem 3.5 below). Later in Section \ref{section_ratio}, we will strengthen the conclusion to $\lim_{t\to 0} \rho(t,x)/t=1$, which completes the proof of Theorem \ref{main_local}.

Note that if Theorem \ref{type_3_ok} is true, then the statement also holds if one replace $B_{3/2}(p)$ and $B_1(p)$, by balls $B_{r}(z)$ and $B_{2r/3}(z)$, respectively, where $r\in(0,1)$ and $B_{r}(z)\subset B_{3/2}(p)$. This follows directly from a rescale of the metric.

\begin{thm}\label{main_pre}
	Let $(M_i,p_i)$ be a sequence of Riemannian $n$-manifolds converging to $(X,p)$ in the Gromov-Hausdorff topology with the conditions below:\\
	(1) $B_2(p_i)\cap \partial M_i=\emptyset$ and the closure of $B_2(p_i)$ is compact;\\
	(2) $\mathrm{Ric}\ge -(n-1)$ on $B_2(p_i)$, and $\mathrm{vol}(B_1(p_i))\ge v>0$.\\
	Then for any $x\in B_1(p)$, $\lim_{t\to 0} \rho(t,x)=0$ holds.
\end{thm}

\begin{proof}[Proof of Theorem \ref{main_pre} by assuming Theorem \ref{type_3_ok}]
	Let $x$ be a point in $B_1(p)$. We prove $\lim_{t\to 0}\rho(t,x)=0$ by induction on $\omega(x)$. We have seen in Corollary \ref{large_vol_type_1} that if $\omega(x)>1/2$, then $\lim_{t\to 0}\rho(t,x)=0$. Assuming that $\lim_{t\to 0}\rho(t,x)=0$ holds when $\omega(x)>2^{-k}$, we will prove that it also holds when $\omega(x)>2^{-(k+1)}$, where $k\in \mathbb{N}$.
	
	We choose a small $r>0$ so that $\omega(z)>2^{-(k+1)}$ for all $z\in B_{r}(x)\subset B_{3/2}(p)$ (Remark \ref{rem_omega_lower}). Suppose that $\lim_{t\to 0}\rho(t,x)>0$, by Corollary \ref{type_1_ok} and Theorem \ref{type_3_ok}, then there must be a type II point $z\in B_{r}(x)$ so that $\lim_{t\to 0}\rho(t,z)>0$. Let $z_i \in M_i$ converging to $z$. According to Definition \ref{def_types}, passing to a subsequence if necessary, there is $t_i\to 0$ and small $\epsilon>0$ such that $\rho(t_i,z_i)\ge \epsilon$ for all $i$. In other words, for each $i$ there is a loop $\gamma_i$ contained in $B_{t_i}(z_i)$ that is not contractible in $B_\epsilon(z_i)$. Without lose of generality, we can assume that $\gamma_i$ is based at $z_i$ and has length $\le 4t_i$. As we did in the proof of Lemma \ref{large_vol_contra_lem}(1), we consider $(U_i,y_i)$ as the universal covering space of $(B_\epsilon(z_i),z_i)$ with covering group $H_i=\pi_1(B_\epsilon(z_i),z_i)$. Let $\Gamma_i:=\langle\gamma_i\rangle$. The sequence $(U_i,y_i)$ may not be precompact in Gromov-Hausdorff topology (see \cite[Example 3.2]{SW2}), so we will take a sequence of small balls around $y_i$, which always has precompactness. We consider the pseudo-group action of $H_i$ on $B_{\epsilon/2}(y_i)\subset U_i$. Passing to a subsequence if necessary, we obtain the equivariant Gromov-Hausdorff convergence below:
	\begin{center}
		$\begin{CD}
		(B_{\epsilon/2}(y_i),y_i,\Gamma_i,H_i) @>GH>> 
		(B_{\epsilon/2}(y),y,G,H)\\
		@VV\pi_i V @VV\pi V\\
		(B_{\epsilon/2}(z_i),z_i) @>GH>> (B_{\epsilon/2}(z),z),
		\end{CD}$
	\end{center}
    where $H$ acts isometrically and $B_{\epsilon/2}(z)$ is isomorphic to $B_{\epsilon/2}(y)/H$ \cite{FY92}.
	As seen in Lemma \ref{large_vol_contra_lem}, $\# \Gamma_i\le N$ for some $N$ and $\mathrm{diam}( \Gamma_i\cdot y_i)\to 0$. Consequently, $G$ is a finite group whose action fixes $y$. Let $F_i$ be the fundamental domain centered at $y_i$. For any $s\in (0,\epsilon/4)$, we estimate:
	\begin{align*}
		B_{s+4Nt_i}(y_i)&\ge \sum_{\gamma\in \Gamma_i} \mathrm{vol}(\gamma\cdot(B_s(y_i)\cap F_i))\\
		&= \# \Gamma_i\cdot \mathrm{vol}(B_s(y_i)\cap F_i)\\
		&\ge 2\cdot \mathrm{vol}(B_s(z_i)).
	\end{align*}
	By volume convergence \cite{Co,CC2},
	$$\mathcal{H}^n(B_s(y))\ge 2\cdot \mathcal{H}^n (B_s(z)).$$
	Thus
	$$\omega(y)=\lim\limits_{s\to 0}\dfrac{\mathcal{H}^n(B_s(y))}{\mathrm{vol}(B^n_s(0))}\ge \lim\limits_{s\to 0}\dfrac{2\cdot \mathcal{H}^n(B_s(z))}{\mathrm{vol}(B^n_s(0))}= 2\cdot \omega(z) > 2^{-k}.$$
	By the induction assumption, we deduce that $\lim_{t\to 0}\rho(t,y)=0$.
	
	For any loop $c$ in $B_t(z)$ based at $z=\pi(y)$, connecting $z$ to a point in $c$ by a minimal geodesic if necessary, we can assume that $c$ is based at $z$. We lift $c$ to a curve $\overline{c}$ starting at $y$ \cite[Chapter II Theorem 6.2]{Bre}. Note that the orbit $H\cdot y$ is discrete. Let $d>0$ the distance between $y$ and its closet point in the orbit $H\cdot y$. For $t<d/4$, the lift $\bar{c}$ is contained in $B_{t}(H\cdot y)$; hence it must be contained in $B_t(y)$ because $\bar{c}$ is a curve. Since $y$ is the only orbit point in $B_t(y)$, the lift $\overline{c}$ ends at $y$ as well. Moreover, because $\overline{c}$ is contained in $B_t(y)$, for any $\sigma>0$ there is a homotopy between $\overline{c}$ and a trivial loop with image in $B_{\rho(t,y)+\sigma}(y)$. Projecting this homotopy down to $B_\epsilon(z)$, we obtain a desired homotopy from $c$ to a trivial loop with image in $B_{\rho(t,y)+\sigma}(z)$. Since $\sigma$ and loop $c$ in $B_t(z)$ is arbitrary, it follows that
	$$\limsup\limits_{t\to 0}\rho(t,z)\le \limsup\limits_{t\to 0}\rho(t,y)=0.$$
	This completes the proof.
\end{proof}

For the rest of this section, we prove some lemmas on the type III points, which will be used in the next section. We always assume the hypothesis of Theorem \ref{type_3_ok} without mentioning, that is, we assume that $\lim_{t\to 0}\rho(t,x)=0$ holds for all type II points.

For any $L>0$, we define
$$\Omega(L)=\{x\in \overline{B_{1.2}(p)}| \limsup_{t\to 0}\rho(t,x)>L\},$$
where $\overline{B_{1.2}(p)}$ means the closure of $B_{1.2}(p)$.
It follows from the hypothesis and Corollary \ref{type_1_ok} that $\Omega(L)$ is a subset of the set of type III points in $\overline{B_{1.2}(p)}$. If $\Omega(L)$ is empty for all $L>0$, then the conclusion of Theorem \ref{type_3_ok} clearly holds. Hence we will assume that $\Omega(L)$ is non-empty for some $L>0$.

\begin{lem}\label{equal_contra_3_lem}
	Let $w$ be an accumulation point of $\Omega(L)$. Then
	$$\limsup\limits_{t\to0 }\rho(t,w)\ge L.$$
	In particular, $w$ is of type III.
\end{lem}

\begin{proof}
	Let $\{z_j\}_j$ be a sequence in $\Omega(L)$ converging to a point $w$. Suppose that
	$$l_w:=\limsup\limits_{t\to0 }\rho(t,w)< L.$$
	Let $\epsilon>0$ with $l_w+2\epsilon<L$. Let $\delta\in (0,\epsilon)$ such that
	$$\rho(t,w)\le l_w+\epsilon$$
	for all $t\in[0,\delta]$. For $z_j$ with $d_j=d(z_j,w)\le \delta/2$ and $t\le \delta/2$, we have
	$$\rho(t,z_j)\le \rho(t+d_j,w)+d_j\le \rho(\delta,w)+\delta/2\le l_w+2\epsilon.$$
	This shows that for $j$ large,
	$$\limsup\limits_{t\to 0}\rho(t,z_j)\le l_w+2\epsilon<L.$$
	A contradiction.
	
	By Corollary \ref{type_1_ok} and the hypothesis on type II points, $w$ must be type III.
\end{proof}

\begin{lem}\label{equal_contra_3}
	Let $L>0$. Suppose that $$d_{GH}(B_2(p_i),B_2(p))=\epsilon_i\to 0.$$
	Then the family of functions
	$$\bigcup\limits_{i=1}^\infty\{\rho(t,z_i)|z_i\in {B_{3/2}(p_i)} \text{ with } d(z_i,z)\le\epsilon_i \text{ for some } z\in {\Omega}(L)\}$$
	is equally continuous at $t=0$. Consequently, there is an indicatrix $\Lambda(t)$ on $[0,T)$ such that $\rho(t,w)\le \Lambda(t)$ for all $t\in[0,T)$ and all $\rho(t,w)$ in the family. 
\end{lem}

\begin{proof}
	Suppose the contrary, then there is $\epsilon_0>0$, $t_j\to 0$, and $z_{i(j)}\in {B_{3/2}(p_i)}$ such that
	$$\rho(t_j,z_{i(j)})\ge \epsilon_0,$$
	where each $z_{i(j)}$ is $\epsilon_{i(j)}$-close to a point $q_j$ in ${\Omega}(L)$. Passing to a subsequence if necessary, we assume that $q_j$ converges to $q\in \overline{\Omega}(L)$, the closure of $\Omega(L)$. Since $z_{i(j)}\to q$ as $j\to\infty$, it follows that $q$ is of type II, which is a contradiction to Lemma \ref{equal_contra_3_lem}.
\end{proof}

\begin{lem}\label{control_annulus}
	Let $\sigma>0$. For a small $r>0$, we consider an annulus of $\Omega(L)$:
	$$A({\Omega}(L);r,1.1)=\{ y\in B_{1.1}(p)\ |\ d(y,{\Omega}(L))\ge r\}.$$
	Then there is $T(r,\sigma)>0$ such that
	$$\rho(T(r,\sigma),y)< L+\sigma$$
	for all $y\in A({\Omega}(L);r,1)$. 
\end{lem}

\begin{proof}
	Since any $y\in  A:=A({\Omega}(L);r,1.1)$ is outside $\Omega(L)$, clearly
	$$\limsup\limits_{t\to0}\rho(t,y)\le L.$$
	Thus for any fixed $y$, there is $T_y>0$ so that
	$$\rho(T_y,y)< L+\sigma.$$
	To find a uniform $T>0$ that works for all $y\in A$, we argue by contradiction. Suppose that there are sequences $T_j\to 0$ and $y_j\in A$ satisfying
	$$\rho(T_j,y_j)\ge L+\sigma.$$
	Passing to a subsequence if necessary, $y_j$ converges to a point $w$. Note that 
	$$d_j+\rho(T_j+d_j,w)\ge \rho(T_j,y_j)\ge L+\sigma,$$
	where $d_j=d(y_j,w)\to 0$. It follows that
	$$\limsup\limits_{t\to 0}\rho(t,w)\ge L+\sigma>L.$$
	In other words, $w\in\Omega(L)$, which contradicts the fact that all $y_j$ are at least distance $r$ away from $\Omega(L)$ for all $j$.
\end{proof}

\section{Constructing homotopies}\label{section_homotopy}

We first roughly explain our approach to prove Theorem \ref{type_3_ok}. Lemma \ref{equal_contra_3} says that for points $z_i\in X_i$ that is close to $z\in \Omega(L)$, we have uniform control on module of $1$-contractibility of $z_i$; Lemma \ref{control_annulus} says that for a limit point $y$ away from $\Omega(L)$, we have control on $\rho(t,y)$. For a loop $c$ contained in a small ball of $x$, using the method in Lemma \ref{retract_homotopy}, we can construct a continuous map $H: K_1^1\to X$ defined on a $1$-skeleton of $D$ extending $c$. Ideally, for a triangle $H(\partial\Delta)$ outside $\Omega(L)$, we wish to extend $H$ over this triangle right away by Lemma \ref{control_annulus}; for a triangle inside $\Omega(L)$, we wish to construct the homotopy from the sequence by utilizing Lemma \ref{equal_contra_3}. In other words, part of the desired homotopy comes from the limit space, while the other part comes from the sequence. When combining these two procedures together to construct the homotopy, we also need to be cautious to assure that we end in a continuous map with controlled size. In practice, we will indeed consider a sequence $\Omega(L_j)$ instead of a single $\Omega(L)$, where $L_j\to 0$.

Regarding constructing the homotopy from the sequence, recall that the proof of Theorem \ref{convergence_LGC} has a similar fashion (also compare Lemma \ref{equal_contra_3} with the conditions of Theorem \ref{convergence_LGC}). We take a close look at whether the proof of Theorem \ref{convergence_LGC} in Section \ref{section_conv} could be useful here. The strategy illustrated above involves a step determining which procedure to be applied for a triangle $H(\partial\Delta)$ based on its position. Also, eventually the desired homotopy will be an extension of $H:K_1^1\to X$ according to this strategy. On the other hand, recall that in the proof of Theorem \ref{convergence_LGC} in Section \ref{section_conv}, we transferred a homotopy along the sequence $\{X_j\}_{j\ge I}$ and forced a uniform convergence, thus the proof does not involve any image of $1$-skeleton in the limit space. 

Because the method in proving Theorem \ref{convergence_LGC} in Section \ref{section_conv} is not compatible with our strategy here, this motivates us to write a new proof of Theorem \ref{convergence_LGC}, which has the advantage that we can keep the image of $1$-skeletons in the limit space at each step. In this alternative proof, we will construct the desired homotopy by defining it on finer and finer skeletons of $D$. As indicated, this method also constitutes part of our proof of Theorem \ref{type_3_ok}. 

We now establish a lemma on constructing a homotopy through refining skeletons. This is similar to \cite[5.1]{Per2}.

\begin{lem}\label{conv_on_skelies}
	Let $D$ be the closed unit disk and let $(X,p)$ be a metric space with the closure of $B_2(p)$ being compact. Let $\Sigma_j$ be a sequence of finite triangular decompositions of $D$ with the conditions below:\\
	(1) each $\Sigma_{j+1}$ is a refinement of $\Sigma_j$,\\
	(2) $\mathrm{diam}(\Delta)\le j^{-1}$ for every triangle $\Delta$ of $\Sigma_j$.\\
	Suppose that we have a sequence of continuous maps $G_j:K_j^1\to B_1(p)$, where $K_j^1$ is the $1$-skeleton of $\Sigma_j$, such that for all $j\ge 1$,\\
	(3) $G_{j+1}|_{K^1_{j}}=G_{j}$,\\
	(4) for any $z\in K^1_{j+1}-K^1_{j}$, 
	$d(G_{j+1}(z),G_j(u))\le 2^{-j}$ holds for all $u$ in the boundary of $\Delta$, where $\Delta$ is a triangle of $\Sigma_j$ containing $z$.\\
	Then $\{G_j\}_j$ converges to a continuous map $G_\infty:D\to B_2(p)$.
\end{lem}

\begin{proof}	
	Clearly we can define
	$$G_\infty:\cup_{k=1}^\infty K_k^1 \to B_1(p)$$
	by setting $G_\infty(z)=G_j(z)$, where $z\in K_j^1$.
	It suffices to show that $G_\infty$
	is uniform continuous. If true, then we can extend $G_\infty$ continuously over $D$.
	
	First notice that condition (4) implies that if $u,v\in \partial \Delta$, where $\Delta$ is a triangle of $\Sigma_j$, then
	$$d(G_j(u),G_j(v))\le 2^{-j+1};$$
	also, if $z\in \Delta \cap (\cup_{k=1}^\infty K_k^1)$, then
	$$d(G_\infty(z),G_j(u))\le \sum_{k=j}^{\infty} 2^{-k}\le 2^{-j+1}.$$
	
	Let $\epsilon>0$. We choose a large integer $J$ so that $2^{-J+1}\le\epsilon/3$. From the triangular decomposition $\Sigma_J$, we construct an open cover $\mathcal{U}$ of $D$ as follows. For every triangle $\Delta$ of $\Sigma_J$, we choose a connected open neighborhood $U_\Delta$ of $\Delta$ such that all vertices in $U_\Delta$ belong to $\partial \Delta$. In this way, if $y\in U_\Delta$, then there is $u\in \partial \Delta$ so that $u$ and $y$ lies in a common triangle of $\Sigma_J$ (this common triangle may not be $\Delta$). Let $\mathcal{U}$ be the collection of all these $U_\Delta$, and let $\delta>0$ be a Lebesgue number of the open cover $\mathcal{U}$. For any $y_1,y_2\in \cup_{k=1}^\infty K_k^1$, if $d(y_1,y_2)<\delta$, then there is a triangle $\Delta$ of $\Sigma_J$ so that $y_1,y_2\in U_\Delta$. Let $u_i\in\partial \Delta$ such that $u_i$ and $y_i$ lies in a common triangle of $\Sigma_J$ $(i=1,2)$. Then
	\begin{align*}
	&d(G_\infty(y_1),G_\infty(y_2))\\
	\le& d(G_\infty(y_1),G_J(u_1))+d(G_J(u_1),G_J(u_2))+d(G_J(u_2),G_\infty(y_2))\\
	\le& 2^{-J+1}+2^{-J+1}+2^{-J+1}\le\epsilon.
	\end{align*}
	This completes the proof.
\end{proof}

We can use the construction in proof of Lemma \ref{retract_homotopy}(ii) to define an extension on the $1$-skeleton, with the distance estimate below.

\begin{lem}\label{rem_extend_1_skele}
	Let $(X,x)$ and $(Y,y)$ be two length metric spaces. Let $c(t)$ and $c'(t)$ be two loops in $X$ and $Y$, respectively. Suppose that\\
	(1) the closure of $B_2(p)$ is compact, where $p=x$ or $y$;\\
	(2) $d_{GH}((X,x),(Y,y))\le \epsilon$,\\
	(3) $c$ and $c'$ are $5\epsilon$-close,\\
	(4) $c'\subset B_t(y)$,\\
	(5) $c$ is contractible in $B_\rho(x)$ via a homotopy $H$.\\
	Then there is are triangular decomposition $\Sigma$ of $D$ and a continuous map $$H':K^1\to Y,$$ where $K^1$ is the $1$-skeleton of $D$, so that\\
	(E1) $d(H'(z),y)\le \rho+18\epsilon$ for all $z\in K^1$.
	
	If condition (5) is replaced by\\
	(5') $c$ is contractible in $B_\rho(c(0))$ via a homotopy $H$,\\
	then we have\\
	(E2) $d(H'(z),H'(u))\le\rho+\mathrm{diam}(c)+32\epsilon$, for all $z\in K^1-\partial D$ and $u\in \partial D$.
\end{lem}

\begin{proof}
	We choose a triangular decomposition $\Sigma$ of $D$ so that $\mathrm{diam}(H(\Delta))\le\epsilon$ for every simplex $\Delta$ of $\Sigma$.
	We follow the method in the proof of Lemma \ref{retract_homotopy}(2) to construct $H'$: first define $H'$ on $\partial D$ by sending $c(t)$ to $c'(t)$, then define $H'$ on $0$-skeletons by mapping any vertex to a near by point, then on edges by minimal geodesics. Then from the proof of Lemma \ref{retract_homotopy}(2), we have
	$$\mathrm{diam}(H'(\partial\Delta))\le 15\epsilon,$$
	for all triangle $\Delta$ of $\Sigma$.
	When $z\in \partial D$, then clearly
	$$d(H(z),H'(z))\le5\epsilon.$$
	When $z\in K^1-\partial D$, let $\Delta$ be a triangle containing $z$ and let $v$ be a vertex connected to $z$, then
	\begin{align*}
	d(H(z),H'(z))&\le d(H(z),H(v))+d(H(v),H'(v))+d(H'(v),H'(z))\\
	&\le \epsilon+\epsilon+15\epsilon=17\epsilon.
	\end{align*}
	Also, for any $z\in K^1-\partial D$, we have \textit{(E1)}:
	\begin{align*}
	d(H'(z),y)&\le d(H'(z),H(z))+d(H(z),x)+d(x,y)\\
	&\le 17\epsilon+\rho+\epsilon\\
	&=\rho+18\epsilon.
	\end{align*}	
	If alternatively, $c$ is contractible in $B_\rho(c(0))$ instead of $B_\rho(x)$, then for any $z\in K^1-\partial D$ and $u\in \partial D$, we have \textit{(E2)}:
	\begin{align*}	
	d(H'(z),H'(u))&\le d(H'(z),H(z))+d(H(z),c(0))+d(c(0),c'(0))+d(c'(0),H'(u))\\
	&\le 17\epsilon+\rho+5\epsilon+\mathrm{diam}(c')\\
	&\le \rho+22\epsilon+(\mathrm{diam}(c)+10\epsilon)\\
	&\le \rho+\mathrm{diam}(c)+32\epsilon.
	\end{align*}
\end{proof}

We use Lemmas \ref{conv_on_skelies} and \ref{rem_extend_1_skele} to present an alternative proof of Theorem \ref{convergence_LGC}.

\begin{proof}[Alternative proof of Theorem \ref{convergence_LGC}.]
	Let $\sigma>0$. Let $q\in B_{3/2}(y)$ and $c$ be a loop in $B_t(q)$, where $t\in(0,T)$. Let $\epsilon_i\to 0$ be a decreasing sequence with $\epsilon_1\le T/20$, $\epsilon_{i+1}\le \epsilon_i/20$, and $\lambda(20\epsilon_i)+20\epsilon_i\le 2^{-i}$ for all $i$.
	
	We fix a large integer $I$ so that 
	$$t+10\epsilon_I<T,\quad \lambda(t+10\epsilon_I)+18\epsilon_I+(2+2\sigma)2^{-I}\le 1/2.$$
	Let $c_I$ be a loop in $X_I$ that is $5\epsilon_I$-close to $c$. By hypothesis, $c_I$ is contractible in $B_{\lambda(t+10\epsilon_I)+\sigma2^{-I}}(c_I(0))$ via a free homotopy $H_I$. Following the procedure and estimates in Lemma \ref{rem_extend_1_skele}, there is a triangular decomposition $\Sigma_1$ of $D$ and a continuous map, which corresponds to $H'$ as in Lemma \ref{rem_extend_1_skele},
    $$G_1:K_1^1\to Y$$ with the properties below:\\
	\textit{(1A)} $\mathrm{diam}(H_I(\Delta))\le \epsilon_I$, $\mathrm{diam}(\Delta)\le I^{-1}$, $\mathrm{diam}(G_1(\partial\Delta))\le 15\epsilon_I$ for any triangle $\Delta$ of $\Sigma_1$;\\
	\textit{(1B)} $d(G_1(z),q)\le \lambda(t+10\epsilon_I)+\sigma2^{-I}+18\epsilon_I$ for any $z\in K_1^1$ (Lemma \ref{rem_extend_1_skele}\textit{(E1)}).
	 
	Next we apply the same procedure to every loop $G_1(\partial\Delta)$, where $\Delta$ is any triangle of $\Sigma_1$. Let $c_{\Delta}$ be a loop in $X_{I+1}$ that is $5\epsilon_{I+1}$-close to $G_1(\partial\Delta)$. Then
	$$\mathrm{diam}(c_\Delta)\le 10\epsilon_{I+1}+\mathrm{diam}(G_1(\partial\Delta))\le 10\epsilon_{I+1}+15\epsilon_I<16\epsilon_I.$$
	By assumption, $c_{\Delta}$ is contractible in $B_{\lambda(16\epsilon_I)+\sigma2^{-I-1}}(c_\Delta(0))$ with a homotopy $H_{\Delta}$. The same procedure provides $\Sigma_{2,\Delta}$, a triangular decomposition of $\Delta$, and a continuous map
	$$G_{2,\Delta}:K^1_{2,\Delta}\to Y$$
	such that\\
	\textit{(2A)} $\mathrm{diam}(H_{\Delta}(\Delta'))\le \epsilon_{I+1}$, $\mathrm{diam}(\Delta')\le (I+1)^{-1}$, $\mathrm{diam}(G_1(\partial\Delta'))\le 15\epsilon_{I+1}$ for any triangle $\Delta'$ of $\Sigma_{2,\Delta}$;\\
	\textit{(2B)} $G_{2,\Delta}|_{\partial\Delta_1}=G_1|_{\partial\Delta_1}$, and by Lemma \ref{rem_extend_1_skele}\textit{(E2)},
	\begin{align*}
		d(G_{2,\Delta}(z),G_1(u))&\le \lambda(16\epsilon_I)+\sigma2^{-I-1}+ \mathrm{diam}(G_1(\partial \Delta))+32\epsilon_{I+1} \\ 
		&\le \lambda(16\epsilon_I)+17\epsilon_I+\sigma2^{-I-1}\\
		&\le 2^{-I}+\sigma2^{-I-1}=2^{-I}(1+\sigma/2)
	\end{align*}
    for all $z\in K^1_{2,\Delta}$ and $u\in\partial\Delta$.
	Since this can be done for each triangle $\Delta$ of $\Sigma_1$, we obtain $\Sigma_2$, a triangular decomposition of $D$ which refines $\Sigma_1$,  and a continuous map $G_2:K_2^1\to Y$ such that
	$$d(G_{2}(z),G_1(u))\le 2^{-I}(1+\sigma/2)$$
	holds for all $z\in K_2^1-K_1^1$ and all $u\in \partial\Delta$, where $\Delta$ is a triangle of $\Sigma_1$ containing $z$.
	
	Repeating this process, we result in a sequence of triangular decomposition $\Sigma_j$ and a sequence of continuous maps $G_j:K_j^1\to Y$ such that\\
	\textit{(jA)} $\mathrm{diam}(\Delta)\le (I+j-1)^{-1}$, and each $\Sigma_{j+1}$ is a refinement of $\Sigma_{j}$;\\
	\textit{(jB)} $G_{j+1}|_{K^1_j}=G_j$, and
	$$d(G_{j+1}(z),G_j(u))\le 2^{-(I+j-1)}(1+\sigma/2),$$
	for all $z\in K_{j+1}^1-K_j^1$ and all $u\in \partial\Delta$, where $\Delta$ is a triangle of $\Sigma_j$ containing $z$.
	
	By Lemma \ref{conv_on_skelies}, $G_j$ converges to a continuous map $G_\infty: D\to Y$, which realizes the homotopy between $c$ and a trivial loop. Moreover, by \textit{(1B)} and \textit{(jB)}, we have
	\begin{align*}
	d(G_\infty(z),q)&\le \lambda(t+10\epsilon_I)+\sigma2^{-I}+18\epsilon_I+\sum_{j=1}^\infty (1+\sigma/2)2^{-(I+j-1)}\\
	&< \lambda(t+10\epsilon_I)+20\epsilon_I+(2+2\sigma)2^{-I}
	\end{align*}
	for all $z\in D$. In other words,
	$$\mathrm{im}(G_\infty) \subset B_{\lambda(t+10\epsilon_I)+20\epsilon_I+(2+2\sigma)2^{-I}}(q).$$ 
	Since $I$ can be chosen arbitrarily large and loop $c$ is also arbitrary in $B_t(x)$, we conclude that
	$\rho(t,q)\le \lambda(t)$.
\end{proof}

As indicated, when constructing the homotopy in proving Theorem \ref{type_3_ok}, we will extend some of the pieces directly in the limit space by using Lemma \ref{control_annulus}. To accommodate this procedure, we modify Lemma \ref{conv_on_skelies} as below so that we can fill in some of the triangles at every step.

\begin{lem}\label{ext_plus_conv}
	Let $D$ be the closed unit disk and let $(X,p)$ be a metric space with the closure of $B_2(p)$ being compact. Let $L_j\to 0$ be a sequence of positive numbers. Let $(\Sigma_1,G_1,\mathcal{E}_1,F_1)$ be a quadruple defined as below:\\
	\textit{(1A)} $\Sigma_1$, a finite triangular decomposition of $D$;\\
	\textit{(1B)} $G_1:K^1_1\to B_1(p)$, a continuous map defined on the $1$-skeleton of $\Sigma_1$;\\
	\textit{(1C)} $\mathcal{E}_1$, a collection of some triangles of $\Sigma_1$;\\
	\textit{(1D)} $F_1:\mathcal{E}_1\to B_1(p)$, a continuous extension of $G_1|_{\mathcal{E}^1_1}$, where ${\mathcal{E}^1_1}$ means the $1$-skeleton of $\mathcal{E}_1$, such that for each triangle $\Delta$ in $\mathcal{E}_1$, there is a point $u\in\partial\Delta$ so that
	$d(F_1(z),G_1(u))\le L_1$ holds for all $z\in \Delta$.\\
	Suppose that we have inductively defined $(\Sigma_j,G_j,\mathcal{E}_j,F_j)$ for each $j\ge 2$:\\
	\textit{(jA)} $\Sigma_j$, a finite triangular decomposition of $\mathcal{C}_{j-1}$, where $\mathcal{C}_{j-1}=\overline{D-\cup_{k=1}^{j-1}\mathcal{E}_k}$, such that $\Sigma_j$ refines $\Sigma_{j-1}|_{\mathcal{C}_{j-1}}$ and
	$\mathrm{diam}(\Delta)\le j^{-1}$
	for any triangle $\Delta$ of $\Sigma_j$;\\
	\textit{(jB)} $G_j:K_j^1\to B_1(p)$ such that $G_j|_{K_{j-1}^1}=G_{j-1}$; also, for any triangle $\Delta$ of $\Sigma_{j-1}$ in $\mathcal{C}_{j-1}$, $d(G_j(z),G_{j-1}(u))\le 2^{-j}$ holds for all $z\in \Delta \cap (K^1_j-K^1_{j-1})$ and all $u\in\partial\Delta$;\\
	\textit{(jC)} $\mathcal{E}_j$, a collection of some triangles of $\Sigma_j$;\\
	\textit{(jD)} $F_j:\mathcal{E}_j\to B_1(p)$ a continuous extension of $G_j|_{\mathcal{E}^1_j}$ such that for each triangle $\Delta$ in $\mathcal{E}_j$, there is $u\in\partial\Delta$ so that
	$d(F_j(z),G_j(u))\le L_j$ holds for all $z\in \Delta$.\\
	Then there is a continuous map $H:D\to B_2(p)$ that extends $G_j$ and $F_j$ for all $j$.
\end{lem}

\begin{proof}
	It is clear that at each step, $\{G_k\}_{k=1}^j$ and $\{F_k\}_{k=1}^j$ form a continuous map
	$$H_j: \cup_{k=1}^j (K_k^1\cup \mathcal{E}_k)\to B_1(p).$$
	We show that this sequence $\{H_j\}$ naturally defines a uniform continuous map
	$$H_\infty: \cup_{j=1}^\infty (K_j^1\cup \mathcal{E}_j)\to B_1(p).$$
	If true, then $H_\infty$ extends continuously over $D$.
	
	Let $\epsilon>0$. Choose an integer $J$ so that $2^{-J}\le \epsilon/16$ and $L_j\le \epsilon/8$ for all $j\ge J$. From $\cup_{k=1}^J \Sigma_k$, a triangular decomposition of $D$, we construct an open cover $\mathcal{U}$ of $D$ as we did in the proof of Lemma \ref{conv_on_skelies}. Let $\delta_1>0$ be a Lebesgue number of $\mathcal{U}$. Let $\delta_2>0$ so that
	$$d(H_J(y_1),H_J(y_2))\le \epsilon/4$$
	for all $y_1,y_2\in \cup_{k=1}^J (K_k^1\cup \mathcal{E}_k)$ with $d(y_1,y_2)\le \delta_2$. We put $\delta=\min\{\delta_1,\delta_2\}$.
	
	Let $\Delta_J$ be any triangle of $\Sigma_J$ in $\mathcal{C}_{J-1}$. We claim that 
	$$d(H_\infty(y),G_J(v))\le\epsilon/4$$
	holds for any $y\in \Delta_J \cap(\cup_{k=J}^\infty (K_k^1\cup \mathcal{E}_k))$ and any $v\in\partial \Delta_J$. In fact, there are two cases on how $H_\infty(y)$ is defined:\\
	\textit{Case 1.} $H_\infty(y)$ is defined as $G_j(y)$ for some $j\ge J$. Then by condition \textit{(jB)}, for any $v\in\partial\Delta_J$,
	$$d(H_\infty(y),G_J(v))=d(G_j(y),G_J(v))\le \sum_{k=J}^\infty 2^{-k}\le 2^{-J+1}\le\epsilon/8.$$
	\textit{Case 2.} $H_\infty(y)$ is defined as $F_j(y)$ for some $j\ge J$. By condition \textit{(jD)},
	$$d(F_j(y),G_j(u))\le L_j$$
	for some $u\in\partial \Delta_j$, where $\Delta_j$ is a triangle of $\Sigma_j$ containing $y$. If $j=J$, then clearly for all $v\in \partial\Delta_J$,
	\begin{align*}
	d(H_\infty(y),G_J(v))=d(H_J(y),G_J(v))+d(G_J(v),G_J(u))\le 2L_J\le\epsilon/4;
	\end{align*}
	if $j>J$, then by condition \textit{(jB)}, for any $v\in\partial\Delta_J$, we have
	\begin{align*}
	d(H_\infty(y),G_J(v))&\le d(H_j(y),H_j(u))+d(G_j(u),G_J(v))\\ &\le L_j+\sum_{k=J}^{j-1} 2^{-k}\le \epsilon/8+2^{-J+1}\le\epsilon/4.
	\end{align*}
    This verifies the claim. 

    Let $y_1,y_2\in \cup_{j=1}^\infty (K_j^1\cup \mathcal{E}_j)$ with $d(y_1,y_2)\le\delta$. If both $y_1$ and $y_2$ belong to $\cup_{k=1}^J \mathcal{E}_k$, then 
    $$d(H_\infty(y_1),H_\infty(y_2))=d(H_J(y_1),H_J(y_2))\le\epsilon/4$$
    because $\delta\le \delta_2$. If both $y_1$ and $y_2$ belong to $\mathcal{C}_J=\overline{D-\cup_{k=1}^{J}\mathcal{E}_k}$, since $\delta\le \delta_1$, there is a triangle $\Delta_J$ of $\Sigma_J$ with $y_1,y_2\in U_{\Delta_J}$. Let $v_i\in\partial \Delta_J$ so that $v_i$ and $y_i$ lies in a common triangle of $\Sigma_J$ $(i=1,2)$. Then by the claim we have shown,
    \begin{align*}
    	&d(H_\infty(y_1),H_\infty(y_2))\\
    	\le& d(H_\infty(y_1),G_J(v_1))+d(G_J(v_1),G_J(v_2))+d(G_J(v_2),H_\infty(y_2))\\
    	\le& \epsilon/4+\max\{2L_J, 2^{-J+1}\}+\epsilon/4\le3\epsilon/4.
    \end{align*}
    Finally, if $y_1\in \cup_{k=1}^J \mathcal{E}_k$ while $y_2\in \mathcal{C}_J$, there is a point $v$ on the segment from $y_1$ to $y_2$ so that $v\in \mathcal{C}_J\cap (\cup_{k=1}^J \mathcal{E}_k)$. Since $d(y_1,v)\le\delta$ and $d(y_2,v)\le\delta$, we see that
    $$d(H_\infty(y_1),H_\infty(y_2))\le d(H_\infty(y_1),H_\infty(v))+d(H_\infty(v),H_\infty(y_2))\le\epsilon.$$
    We complete the proof of uniform continuity.
\end{proof}

We prove Theorem \ref{type_3_ok} by using Lemmas \ref{equal_contra_3}, \ref{control_annulus}, and \ref{ext_plus_conv}.

\begin{proof}[Proof of Theorem \ref{type_3_ok}]
    Let $x$ be any type III point in $B_1(p)$. Let
    $$\epsilon_i:=d_{GH}(B_2(p_i),B_2(p))\to 0$$
    and let $x_i\in M_i$ be a sequence converging to $x$ with $d(x_i,x)\le \epsilon_i$. Since $x$ is of type III, by definition there is an indicatrix $\lambda(t)$ on $[0,T)$ such that $\rho(t,x_i)\le \lambda(t)$ for all $i$ and all $t\in [0,T)$. Shrinking $T$ if necessary, we can assume that $\lambda(T)<1/40$. Suppose that $\lim_{t\to 0}\rho(t,x)=0$ fails for this $x$, then $\Omega(L)$ is non-empty for some $L>0$; we will show that $\rho(t,x)\le\lambda(t)$ for all $[0,T)$, which is a contradiction.
    
    Let $L_j\to 0$ be a decreasing sequence. For each $j$, we define 
    $$\Omega(L_j)=\{w\in \overline{B_{1.2}(p)}| \limsup_{t\to 0}\rho(t,w)>L_j\}.$$
    By Lemma \ref{equal_contra_3}, for each $j$ there is an indicatrix $\Lambda_j(t)$ that bounds all $\rho(t,z)$, where $z\in \overline{B_{1.2}(p_i)}$ is $\epsilon_i$-close to some point $w$ of $\Omega(L_j)$. Since $\lim_{t\to 0}\Lambda_j(t)=0$ for each $j$, we can choose $s_j>0$ with
    $$\Lambda_j(2s_j)+4s_j\le 2^{-j}.$$
    We put $A_j:=A(\Omega(L_j);s_j,1.1)$. By Lemma \ref{control_annulus}, for each $j$, there is $T_j>0$ such that 
    $$\rho(T_j,y)\le L_j+2^{-j}$$
    for all $y\in A_j$. Replacing $T_j$ by a smaller number if necessary, we can assume that
    $$T_j< 2s_j,\quad T_{j+1}\le T_j/20$$ 
    for all $j$.
    We choose a subsequence $i(j)$ so that
    $$\epsilon_{i(j)}=d_{GH}(B_2(p_{i(j)}),B_2(p))\le T_j/20=:\delta_j.$$
    
    Let $\sigma>0$. Fix $t\in (0,T)$ and a loop $c$ contained in $B_t(x)$. We will construct a homotopy between $c$ and a trivial loop with controlled image by using Lemma \ref{ext_plus_conv}. Fix an integer $J\ge1$ so that 
    $$t+6\delta_J<T,\quad \lambda(t+6\delta_J)+\sigma 2^{-J}+18\delta_J<1/20, \quad (1+\sigma) 2^{-J+1}<1/20.$$
    Let $c_{i(J)}$ be a loop in $B_1(x_{i(J)})$ that is $5\delta_J$-close to $c$. Since $c_{i(J)}\subset B_{t+6\delta_J}(x_{i(J)})$, $c_{i(J)}$ is contractible in $B_{\lambda(t+6\delta_J)+\sigma 2^{-J}}(x_{i(J)})$. By the construction in Lemma \ref{rem_extend_1_skele} (also see the procedure in alternative proof of Theorem \ref{convergence_LGC}), there is a triangular decomposition $\Sigma_1$ of $D$ and a continuous map $G_1:K_1^1\to X$ such that\\
    \textit{(1A)} $\mathrm{diam}(\Delta)\le J^{-1}$, $\mathrm{diam}(G_1(\partial\Delta))\le 15\delta_J< T_J$ for any triangle $\Delta$ of $\Sigma_1$;\\
    \textit{(1B)} $d(G_1(z),x)\le \lambda(t+6\delta_J)+\sigma 2^{-J}+18\delta_J$ for any $z\in K_1^1$ (Lemma \ref{rem_extend_1_skele}\textit{(E1)}).\\
    In particular, $\mathrm{im} (G_1)$ belongs to $B_{1/20}(x)\subset B_{1.1}(p)$.
    Note that if a triangle $\Delta$ has a point $u\in\partial\Delta$ so that $G_1(u)\in A_J$, then by \textit{(1A)}, $G_1(\partial\Delta)$ is contractible in $B_{L_J+2^{-J}}(G_1(u))$. In this case we can directly extend $G_1$ continuously over $\Delta$. With this in mind, we consider \\
    \textit{(1C)} $\mathcal{E}_1=\{z\in D|\ z\in\Delta \text{ with } G_1(\partial\Delta)\cap A_J\not=\emptyset \}$.\\
    As explained, over $\mathcal{E}_1$, $G_1$ extends to a continuous map $F_1:\mathcal{E}_1\to X$.\\
    \textit{(1D)} $F_1:\mathcal{E}_1\to X$ satisfies that for any $\Delta\subset \mathcal{E}_1$, there is $u\in\partial\Delta$, such that
    $$d(F_1(z),G_1(u))\le L_J+2^{-J}\le L_J+2^{-J}$$
    holds for all $z\in\Delta$.\\
    This completes the first step in constructing the desired homotopy.
    
    Next we deal with the triangles outside $\mathcal{E}_1$. Let $\Delta$ be a triangle of $\Sigma$ such that $G_1(\partial \Delta)\cap A_J=\emptyset$. This implies $d(z,\Omega(L))<s_J$ for all $z\in G_1(\partial\Delta)$. On $B_1(x_{i(J+1)})$, there is a loop $c_{\Delta}$ that is $5\delta_{J+1}$-close $G_1(\partial\Delta)$. Then
    $$\mathrm{diam}(c_\Delta)\le 10\delta_{J+1}+\mathrm{diam}(G_1(\partial\Delta))\le 10\delta_{J+1}+15\delta_J<16\delta_J.$$
    By our choice of $\mathcal{E}_1$, there is $w\in \Omega(L_J)$ such that $G_1(\partial\Delta)\subset B_{s_J}(w)$. Let $w'$ be a point in $B_1(x_{i(J+1)})$ that is $\epsilon_{i(J+1)}$-close to $w$. Then it is direct to check that
    $$\mathrm{im}(c_\Delta)\subset B_{6\delta_{J+1}+s_J}(w')\subset B_{2s_J}(w').$$
    Therefore, $c_\Delta$ is contractible in $B_{\Lambda_{J}(2s_J)+2^{-J}\sigma}(w')\subset B_{\Lambda_{J}(2s_J)+2^{-J}\sigma+2s_J}(c_\Delta(0))$. By the same construction we have applied before, we obtain a triangular decomposition $\Sigma_{2,\Delta}$ and a continuous map $G_{2,\Delta}:K_2^2\to X$ such that\\
    \textit{(2A)} $\mathrm{diam}(\Delta')\le (J+1)^{-1}$, $\mathrm{diam}(G_{2,\Delta}(\partial\Delta'))\le 15\delta_{J+1}<T_{J+1}$ for any triangle $\Delta'$ of $\Sigma_{2,\Delta}$;\\
    \textit{(2B)} $G_{2,\Delta}|_{\partial\Delta}=G_1|_{\partial\Delta}$, and by Lemma \ref{rem_extend_1_skele}\textit{(E2)},
    \begin{align*}
     d(G_{2,\Delta}(z),G_1(u))&\le \Lambda_{J}(2s_J)+2^{-J}\sigma+2s_J+\mathrm{diam}(G_1(\partial\Delta))+32\delta_{J+1}\\
     &\le \Lambda_{J}(2s_J)+2^{-J}\sigma+2s_J+16\delta_J+32\delta_{J+1}\\
     &\le \Lambda_{J}(2s_J)+4s_J+2^{-J}\sigma\\
     &\le 2^{-J}(1+\sigma)
    \end{align*}
    for all $z\in K_{2,\Delta}^1$ and $u\in \partial\Delta$. Because we can apply the same argument to each triangle $\Delta$ in $\mathcal{C}_1=\overline{D-\mathcal{E}_1}$, we result in $\Sigma_2$, a triangular decomposition of $\mathcal{C}_1$ which refines $\Sigma_1$, and a continuous map $G_2:K_2^1\to B_1(x)$ such that
    $$d(G_{2}(z),G_1(u))\le 2^{-J}(1+\sigma)$$
    holds for all $z\in K_2^1-K_1^1$ and all $u\in \partial\Delta$, where $\Delta$ is a triangle of $\Sigma_1$ in $\mathcal{C}_1$ containing $z$. Note that
    $$d(G_2(z),x)\le d(G_2(z),G_1(u))+d(G_1(u),x)\le 1/20+(1+\sigma)2^{-J}\le 1/10.$$
    Thus $\mathrm{im}(G_2)\subset B_{1.1}(p)$.
    Next we set\\
    \textit{(2C)} $\mathcal{E}_2=\{z\in D|\ z\in\Delta \text{, a triangle of $\Sigma_2$, with } G_2(\partial\Delta)\cap A_{J+1}\not=\emptyset \}$.\\
    By the same argument that implies $\textit{(1D)}$, we can extend $G_2$ continuously over $\mathcal{E}_2$. This produces\\ 
    \textit{(2D)} $F_2:\mathcal{E}_2\to X$ such that for any $\Delta\subset \mathcal{E}_2$, there is $u\in\partial\Delta$ so that
    $$d(F_2(z),G_2(u))\le L_{J+1}+2^{-(J+1)}$$
    holds for all $z\in\Delta$.
    
    We repeat this procedure for each $k\ge 2$. This allows us to construct the quad-ruple $(\Sigma_k,G_k,\mathcal{E}_k,F_k)$ with\\
    \textit{(kA)} $\Sigma_k$, a triangular decomposition of $\mathcal{C}_{k-1}$ that refines $\Sigma_{k-1}|_{\mathcal{C}_{k-1}}$, where $$\mathcal{C}_{k-1}=\overline{D-\cup_{j=1}^{k-1}\mathcal{E}_j};$$
    also, $\mathrm{diam}(\Delta)\le (J+k-1)^{-1}$
    for any triangle $\Delta$ of $\Sigma_k$.\\
    \textit{(kB)}  $G_k:K_k^1\to X$ such that $G_k|_{K_{k-1}^1}=G_{k-1}$; also, for any triangle $\Delta$ of $\Sigma_{k-1}$ in $\mathcal{C}_{k-1}$, $$d(G_k(z),G_{k-1}(u))\le 2^{-(J+k-2)}(1+\sigma)$$ holds for all $z\in \Delta \cap (K^1_k-K^1_{k-1})$ and all $u\in\partial\Delta$;\\
    With \textit{(kB)}, we can check that $\mathrm{im}(G_k)$ does not exceed $B_{1.1}(p)$:
    $$d(G_k(z),x)\le 1/20+\sum_{j=2}^{k}2^{-(J+j-2)}(1+\sigma)\le 1/10.$$
    Then we define the followings.\\
    \textit{(kC)} $\mathcal{E}_k=\{z\in D|\ z\in\Delta \text{, a triangle of $\Sigma_k$, with } G_k(\partial\Delta)\cap A_{J+k-1}\not=\emptyset \}$.\\
    \textit{(kD)} $F_k:\mathcal{E}_k\to X$ such that for any $\Delta\subset \mathcal{E}_k$, there is $u\in\partial\Delta$ so that
    $$d(F_k(z),G_k(u))\le L_{J+k-1}+2^{-(J+k-1)}$$
    holds for all $z\in\Delta$.
    
    Applying Lemma \ref{ext_plus_conv}, we end in a continuous map $H$ extending all $G_k$ and $F_k$. $H$ realizes the homotopy between $c$ and a trivial loop. Moreover, by the above \textit{(1B)}, \textit{(kB)}, and \textit{(kD)}, we see that for any $z\in D$, 
    \begin{align*}
    	d(H(z),x)&\le \lambda(t+6\delta_J)+\sigma 2^{-J}+18\delta_J+\left(\sum_{k=2}^\infty 2^{-(J+k-2)}(1+\sigma)\right)+\left(L_J+2^{-J}\right)\\
    	&\le \lambda(t+6\delta_J)+18\delta_J+L_J+(3\sigma+3)2^{-J}.
    \end{align*}
   Since $J$ can be arbitrarily large, it follows that $\rho(t,x)\le \lambda(t)$. This completes the proof. 
\end{proof}

\section{Proof of $\lim_{t\to 0}\rho(t,x)/t=1$}\label{section_ratio}

Based on Theorem \ref{main_pre}, we finish the proof of Theorem \ref{main_local} in this section by using Sormani's uniform cut technique \cite{Sor}. 

The uniform cut technique is based on Abresh-Gromoll's excess estimate \cite{AG}. Even though excess estimate has been extended to metric spaces with synthetic Ricci curvature bounds (RCD spaces) \cite{GM,MN}, which includes Ricci limit spaces, it is unclear to the authors whether excess estimate holds on a covering space of a local incomplete ball in the limit space. Therefore, we will go back to the sequence, find uniform cut points on the manifolds, then pass them to the limit.

We assume dimension $n\ge 3$ in this section. When $n=2$, the limit space is an Alexandrov space, which is locally contractible.

Our first goal is a localized version of uniform cut theorem with a parameter $\epsilon$ (see Lemma \ref{uniform_cut_manifold}). This is similar to \cite[Section 4]{Wylie}, where the nonnegative Ricci curvature case is considered. We include the complete proof for readers' convenience.

We recall Abresh-Gromoll's excess estimate \cite{AG}:

\begin{thm}\label{excess}
	Given $n\ge 3$, there is $C(n)$ such that the following holds.
	
	Let $M^n$ be a manifold of $\mathrm{Ric}\ge -(n-1)$. Let $x,y_1,y_2\in M$ with $d=d(y_1,y_2)\le 1$. Let $\gamma$ be a unit speed minimal geodesic from $y_1$ to $y_2$. Suppose that
	$$d(x,\gamma(d/2))\le r d,$$ 
	where $r\in(0,1/4]$. Suppose that the closure of $B_{R_i}(y_i)$ ($i=1,2$) is compact, where $R_i=d(x,y_i)+rd$. Then
	$$e(x)\le C(n)r^{\frac{n}{n-1}}d,$$
	where
	$e(x)=d(x,y_1)+d(x,y_2)-d(y_1,y_2)$ is the excess function associated to $y_1,y_2$.
\end{thm}	

\begin{lem}\label{uniform_cut_pre}
	Let $M^n$ ($n\ge 3$) be a manifold of $\mathrm{Ric}\ge -(n-1)$. Let $\gamma$ be a unit speed minimal geodesic in $M$ with length $d\le 1$. Let $x\in M$ be a point with
	$$d(x,\gamma(0))\ge (1/2+\epsilon)d,\quad d(x,\gamma(d))\ge (1/2+\epsilon)d.$$
	Suppose that the closure of $B_{(\frac{1}{2}+2\psi(\epsilon))d}(\gamma(0))$ and $B_{(\frac{1}{2}+2\psi(\epsilon))d}(\gamma(d))$ are compact. Then
	$$d(x,\gamma(d/2))\ge \psi(\epsilon)d,$$
	where
	$$\psi(\epsilon)=\min\left\{\dfrac{1}{4},\dfrac{\epsilon^{\frac{n-1}{n}}}{C_1(n)}\right\}$$
	and $C_1(n)$ is a constant depending on $n$.
\end{lem}	

\begin{proof}
	Suppose that $d(x,\gamma(d/2))<\psi(\epsilon)d$, where $\psi(\epsilon)$ as given in the statement; we will determine the constant $C_1(n)$ in the end.
	Applying Theorem \ref{excess},
	$$e(x)\le  C(n)\cdot\psi(\epsilon)^{\frac{n}{n-1}}d.$$
	On the other hand
	$$e(x)=d(\gamma(0),x)+d(\gamma(d),x)-d\ge 2\epsilon d.$$
	Thus
	$$2\epsilon \le C(n)d\cdot\psi(\epsilon)^{\frac{n}{n-1}}.$$
	It follows that
	$$\psi(\epsilon)\ge \left(\dfrac{2\epsilon}{C(n)}\right)^{\frac{n-1}{n}}=\dfrac{\epsilon^{\frac{n-1}{n}}}{C_1(n)},$$
	where
	$$C_1(n)=(C(n)/2)^{\frac{n-1}{n}}.$$
	As a result, if we choose
	$$\psi(\epsilon)=\min\left\{\dfrac{1}{4},\dfrac{\epsilon^{\frac{n-1}{n}}}{2C_1(n)}\right\},$$
	then 
	$$d(x,\gamma(d/2))\ge\psi(\epsilon)d$$
	holds.
\end{proof}

\begin{lem}\label{uniform_cut_manifold}
	Let $(M^n,p)$ be a manifold of $n\ge 3$, $\mathrm{Ric}\ge -(n-1)$, and the closure $B_2(p)$ being compact. Let $x\in B_1(p)$. Let $\epsilon>0$ and let $\psi(\epsilon)$ be the constant in Lemma \ref{uniform_cut_pre}. Suppose that $\gamma$ is a geodesic loop in $\pi_1(B_R(x),x)$ of length $d$ with the properties below, where $B_R(x)\subset B_1(p)$. \\
	(1) If a loop $\gamma'$ based at $x$ is homotopic to $\gamma$ then $\gamma'$ has length $\ge d$.\\
	(2) $\gamma$ is minimal on both $[0,d/2]$ and $[d/2,d]$.\\
	If $R>(1/2+2\psi(\epsilon))d$, then for any $y\in\partial B_{(1/2+\epsilon)d}(x)$, we have
	$$d(y,\gamma(d/2))\ge \psi(\epsilon)d.$$
\end{lem}

\begin{proof}
	Suppose that $d(y,\gamma(d/2))<\psi(\epsilon)d$. Let $\sigma$ be a minimal geodesic from $\gamma(d/2)$ to $y$. Let $(U,\tilde{x})$ be the universal cover of $(B_R(x),x)$, where $\tilde{x}$ is a lift of $x$ in $U$. We lift $\gamma$ to a minimal geodesic $\tilde{\gamma}$ starting from $\tilde{x}$ in $U$. Since $\mathrm{im}(\sigma)\subset B_R(p)$, we can also lift $\sigma$ to a curve $\tilde{\sigma}$ from $\tilde{\gamma}(d/2)$ to a point $\tilde{y}$. It is clear that
	$$d(\tilde{y},\tilde{\gamma}(0))\ge d(y,x)=(1/2+\epsilon)d, \quad d(\tilde{y},\tilde{\gamma}(d))\ge d(y,x)=(1/2+\epsilon)d.$$
	By Lemma \ref{uniform_cut_pre}, we have
	$$d(\tilde{y},\tilde{\gamma}(d/2))\ge \psi(\epsilon)d,$$
	which is a contradiction since $\tilde{\sigma}$ has length $<\psi(\epsilon)d$.
\end{proof}

Let $x\in B_1(p)$ as in Theorem \ref{main_local} and let $B_R(x)\subset B_1(p)$. From Theorem \ref{main_pre}, $B_R(x)$ is semi-locally simply connected. As a result, the universal cover of $B_R(x)$ exists and $\pi_1(B_R(x),x)$ acts freely and isometrically on the universal cover. We denote this universal cover as $(U,\tilde{x})$. If $\pi_1(B_R(x),x)$ is not trivial, then we define the minimal length
$$d=\min_{g\in \pi_1(B_R(x),x)-\{e\}}d(g\tilde{x},\tilde{x})$$
and a subset $\mathcal{S}$ by
$$\mathcal{S}=\{h\in\pi_1(B_R(x),x) | d(h\tilde{x},\tilde{x})=d\}.$$

Let $h\in\mathcal{S}$ and let $\overline{\gamma}$ be a minimal geodesic from $\tilde{x}$ to $h\tilde{x}$. It is not difficult to show that $\gamma$, the projection of $\overline{\gamma}$ to $B_R(x)$, is a geodesic loop and satisfies a halfway property \cite{Sor}, namely, $\gamma$ is minimal on $[0,d/2]$ and $[d/2,d]$. Ideally, we want to find a sequence of geodesic loops $\gamma_i$ converging to $\gamma$ so that each $\gamma_i$ satisfies the halfway property as well. However, this may not be true because $\gamma_i$ may not be the geodesic loop of minimal length and may not be a short generator. Also, $M_i$ may contain shorter and shorter loops disappearing in the limit. To overcome this, roughly speaking, we will consider all the geodesic loops converging to some element of $\mathcal{S}$, instead of a fixed element $\gamma\in\mathcal{S}$ (see Lemma \ref{halfway}).

For Lemmas \ref{local_surj} and \ref{halfway} below, we assume that $x\in B_1(p)$ as in Theorem \ref{main_local}. For $B_r(x)\subset B_R(x)\subset B_1(p)$, where $d/2<r<R$, we set $\rho=(R-r)/4$. Lemma \ref{local_surj} below can be viewed as a localized version of a proof in \cite{Tu}.

\begin{lem}\label{local_surj}
	For any $i$ sufficiently large, there is a group homomorphism
	$$\Phi_i: \pi_1(B_{r+\rho}(x_i),x_i)\to \pi_1(B_R(x),x).$$
	Moreover, $\Phi_i$ is onto $G(r)$, where $G(r)$ is the subgroup of $\pi_1(B_R(x),x)$ generated by loops based at $x$ and contained in $\overline{B_r(x)}$.
\end{lem}

\begin{proof}
	Since $\lim_{t\to 0}\rho(t,x)=0$ for all $y\in B_1(p)$, by the compactness of $\overline{B_{r+2\rho}(x)}$, we can choose $T>0$ such that
	$$\rho(T,y)\le (R-r)/5<\rho$$
	for all $y\in B_{r+2\rho}(x)$; in other words, any loop in $B_T(y)$ is contractible in $B_{r+3\rho}(x)\subset B_R(x)$. We set $\epsilon_0=T/20$ and $N$ large so that
	$$d_{GH}(B_2(p_i),B_2(p))=\epsilon_i\le \epsilon_0$$
	for all $i\ge N$.
	
	Let $c_i$ be a loop based at $x_i$ with $\mathrm{im}(c_i)\subset B_{r+\rho}(x_i)$. Following the method in Lemma \ref{retract_homotopy}(i), we can construct a loop $c$ based at $x$, as a broken geodesic loop, such that $c$ is $5\epsilon_i$-close to $c_i$ and $$\mathrm{im}(c)\subset B_{r+\rho+6\epsilon_i}(x)\subset B_R(x).$$
	Also, by the proof of Lemma \ref{retract_homotopy}(ii), if $c'_i$, another loop based at $x_i$ contained in $B_{r+\rho}(x_i)$, is homotopic to $c_i$ in $B_{r+\rho}(x_i)$, then the corresponding $c'$ is homotopic to $c$ in $B_R(x)$. It is clear that this defines a group homomorphism
	$$\Phi_i: \pi_1(B_{r+\rho}(x_i),x_i)\to \pi_1(B_R(x),x).$$
	
	It remains to prove that $\Phi_i$ is onto $G(r)$. Let $[c]$ be a generator of $G(r)$. Because $B_R(x)$ is semi-locally simply connected, we can represent $[c]$ by a geodesic loop $c:[0,1]\to \overline{B_r(x)}$ based at $x$. Let $0=t_0,t_1,...,t_k=1$ be a partition of $[0,1]$ so that $c|_{[t_j,t_{j+1}]}$ has length $\le\epsilon$. Choose points $z_{i,j}$ in $B_{r+\epsilon_i}(x)$ that is $\epsilon_i$ close to $c(t_j)$ (we set $z_{i,0}=z_{i,k}=x_i$) and then connect these points by minimal geodesics. With this, we result in a broken geodesic loop $c_i$ in $B_{r+\rho}(x)$. It is direct to check $\Phi_i([c_i])=[c]$ by our construction and Lemma \ref{retract_homotopy}(i).
\end{proof}

\begin{lem}\label{halfway}
	For each $i$, let $\mathcal{S}_i=\Phi^{-1}_i(\mathcal{S})$, that is, the set of all the elements in $\pi_1(B_{r+\rho}(x),x)$ mapped to $\mathcal{S}$ under $\Phi_i$. Let $\gamma_i$ be the unit speed geodesic loop of minimal length among all the geodesic loops representing elements of $\mathcal{S}_i$. Suppose that $\gamma_i$ has length $d_i$. Then\\
	(1) $d_i\to d$;\\
    (2) $\gamma_i$ is minimal on $[0,d_i/2]$ and $[d_i/2,d_i]$.
\end{lem}

\begin{proof}
	(1) Passing to a subsequence if necessary, geodesic loops $\gamma_i$ converge uniformly to a limit geodesic loop $\gamma_\infty$ based at $x$. By construction of $\Phi_i$ and Lemma \ref{retract_homotopy}(i), $$\Phi_i[\gamma_i]=[\gamma_\infty]\in\mathcal{S}$$ for all $i$ large. $\gamma_\infty$ has length
	$l(\gamma_\infty)$ at most $\liminf_{i\to\infty} d_i$.
	Let $\gamma$ be a geodesic loop based at $x$ with length $d$ and $[\gamma]=[\gamma_\infty]$.
	Since $\gamma$ has the shortest length among all non-contractible loops in $\pi_1(B_R(x),x)$, we deduce that
	$$d\le l(\gamma_\infty)\le \liminf_{i\to\infty} d_i.$$
	
	It remains to show that $d\ge \limsup_{i\to\infty} d_i$. We prove this by contradiction. Suppose that there is $\delta>0$ such that $$d+\delta<\limsup_{i\to\infty} d_i.$$ Passing to a subsequence, we have $d+\delta<d_i$ for $i$ large. Further shrinking $\delta$ if necessary, we can assume that any element of $\pi_1(B_R(x),x)$ outside $\mathcal{S}\cup\{e\}$ has length at least $d+\delta$ (we can assume so because $\pi_1(B_R(x),x)$ acts on the universal cover $U$ discretely). For $\gamma$ we chosen above, we can follow Lemma \ref{retract_homotopy}(1) to obtain a sequence of loops $\alpha_i$ in $X_i$ that is $5\epsilon_i$-close to $\gamma$, where $$\epsilon_i=d_{GH}((B_R(x_i),x_i),(B_R(x),x))\le 2^{-i}\to 0.$$
	In particular, $\alpha_i$ converges uniformly to $\gamma$. Since $\gamma$ is contained in $\overline{B_{d/2}(x)}$, we have $\alpha_i$ contained in $\overline{B_{d/2+6\epsilon_i}(x_i)}$. Also, due to our construction, it is clear that $\alpha_i$ has length $l(\alpha_i)\le 3d$. For each $\alpha_i$, we divide $\alpha_i$ into $N_i$ pieces $\alpha_i|_{[t_{i,j},t_{i,j+1}]}$ such that each piece has length between $\delta/4$ and $\delta/2$. Note that $N_i\le 12d/\delta$. Passing to a subsequence, we can assume that all $N_i$ are equal. For each $j=0,..,N-1$, let $\beta_{i,j}$ be a loop joining a minimal geodesic from $x_i$ to $\alpha_i(t_{i,j})$, $\alpha_i|_{[t_{i,j},t_{i,j+1}]}$, and a minimal geodesic from $\alpha_i$ back to $x_i$. $\beta_{i,j}$ has length
	$$l(\beta_{i,j})\le 2(d/2+6\epsilon_i)+\delta/2=d+\delta/2+12\epsilon_i.$$
	As $i\to \infty$, each $\beta_{i,j}$ converges to a loop $\beta_{j}$ based on $x$ with $\beta_j$ has length $\le d+\delta/2$. By the choice of $\delta$, each $\Phi_i([\beta_{i,j}])=[\beta_j]$ either is trivial or belongs to $\mathcal{S}$. If all $[\beta_j]$ are trivial, then
	$$[\gamma]=\Phi_i([\alpha_i])=\prod_{j=0}^{N-1} \Phi([\beta_{i,j}])=\prod_{j=0}^{N-1} [\beta_j]$$
	would be trivial too, a contradiction. Consequently, there must be some $[\beta_j]\in\mathcal{S}$. For such a $[\beta_j]$, we have $[\beta_{i,j}]\in\mathcal{S}_i$,
	but for $i$ large,
	$$l(\beta_{i,j})\le d+\delta/2+12\epsilon_i<d_i=l(\gamma_i).$$
	This is a contradiction to our choice of $\gamma_i$. This completes the proof of (1).
	
	(2) Let $\sigma_i$ be a minimal geodesic from $x$ to $\gamma_i(d_i/2)$. Suppose that $\sigma_i$ has length $<d_i/2$. Let $c_{i,1}$ be the loop joining $\gamma_i|_{[0,d_i/2]}$ and $\sigma_i^{-1}$ and let $c_{i,2}$ be the loop joining $\sigma_i$ and $\gamma_i|_{[d_i/2,d_i]}$. It is clear that $[c_{i,1}][c_{i,2}]=[\gamma_i]$ in $\pi_1(B_{r+\rho}(x_i),x_i)$. Since both $c_1$ and $c_2$ have length strictly shorter than $\gamma_i$, due to our choice of $\gamma_i$, it follows that $\Phi_i[c_1]$ and $\Phi_i[c_2]$ do not belong to $\mathcal{S}$. Each $c_{i,j}$ consists of two geodesics, thus $c_{i,j}$ subconverges $c_{\infty,j}$ with $\Phi_i[c_{i,j}]=[c_{\infty,j}]$, where $j=1,2$. $c_{\infty,j}$ has length $$l(c_{\infty,j})\le \lim_{i\to\infty}d_i=d,$$ where $j=1,2$. However, $[c_{\infty,j}]\notin \mathcal{S}$. Hence $[c_{\infty,j}]=e$, which contradicts with $[c_{\infty,1}][c_{\infty,2}]=[\gamma_\infty]\not=e$.
\end{proof}

With the above preparations, now we are ready to prove Theorem \ref{main_local}, that is, $\lim_{t\to 0}\rho(t,x)/t=1$. Suppose that Theorem \ref{main_local} fails, then we can use Lemmas \ref{halfway} and \ref{uniform_cut_manifold} to find uniform cut points on manifolds, then pass these uniform cut points from the manifolds to a tangent cone of the limit space, which would end in a contradiction to the structure of tangent cones. This argument is a modification of \cite{Sor}.

\begin{proof}[Proof of Theorem \ref{main_local}]
	Fix $x\in B_1(p)$. Suppose that there are $\delta>0$ and $r_j\to 0$ such that $\lim_{j\to\infty}\rho(r_j,x)/r_j\ge1+\delta$. Let $R_j=(1+\delta)r_j$. Then the following holds:\\
	(1) each $B_{R_j}(x)$ is not simply connected,\\
	(2) elements in $\mathcal{S}_j$, the set of shortest nontrivial loops in $\pi_1(B_{R_j}(x),x)$, are represented by loops contained in $B_{r_j}(x)$.
	
	Choose $\epsilon>0$ sufficiently small so that $\epsilon<\psi(\epsilon)<\delta/4$, where $\psi(\epsilon)$ is given in Lemma \ref{uniform_cut_pre}. With this $\epsilon$, 
	$$R_j=(1+\delta)r_j>(1/2+2\psi(\epsilon))d_j,$$
	where $d_j$ is the length of elements in $\mathcal{S}_j$. For each fixed $j$, by Lemma \ref{halfway}, there is $\gamma_{i,j}$ of length $d_{i,j}$ with halfway property and $d_{i,j}\to d_j$ as $i\to\infty$. Let $m_{i,j}=\gamma_{i,j}(d_{i,j}/2)$ be the midpoint of $\gamma_{i,j}$. Since $R_j>(1/2+2\psi(\epsilon))d_{i,j}$ for all $i$ large, applying Lemma \ref{uniform_cut_manifold}, we have
	$$d(y,m_{i,j})\ge \psi(\epsilon)d_{i,j}$$
	for all $y\in \partial B_{(1/2+\epsilon)d_{i,j}}(x)$.
	
	Next we consider the convergence:
	$$(d_j^{-1}B_1(x),x)\overset{GH}\longrightarrow (C_xX,v),$$
	where $C_xX$ is a metric cone with vertex $v$ since $X$ is a non-collapsing Ricci limit space \cite{CC1}. By a standard diagonal argument, we have a convergent subsequence
	$$(d_{i(j),j}^{-1}B_1(x_{i(j)}),x_{i(j)})\overset{GH}\longrightarrow (C_xX,v).$$
	With respect to the above convergence, $m_{i(j),j}\to m\in C_xX$ with $d(m,v)=1/2$.
	
	We claim that there are no rays starting from $v$ and going through $m$, which contradicts the fact that $C_xX$ is a metric cone with vertex $v$. Let $y$ be any point in $\partial B_1(v)$ and let $y_j$ in $d_{i(j),j}^{-1}B_1(x_{i(j),j})$ converges to $y$. Clearly $d(y_j,x)=1+\eta_j$ on $d_{i(j),j}^{-1}B_1(x_{i(j)})$, where $\eta_j\to 0$. Let $z_j$ be the point where a minimal geodesic from $m_j$ to $y_j$ intersects $d_{i(j),j}^{-1}\partial B_{(1/2+\epsilon)d_{i(j),j}}(x_{i(j)})$. Then on $d_{i(j),j}^{-1}B_1(x_{i(j)})$,
	\begin{align*}
		d(y_j,m_{i(j),j})&=d(y_j,z_j)+d(z_j,m_{i(j),j})\\
		&\ge (1+\eta_j)-(1/2+\epsilon)+\psi(\epsilon)\\
		&= 1/2+\eta_j+(\psi(\epsilon)-\epsilon)
	\end{align*}
	Let $j\to\infty$, we see that
	$$d(y,m)\ge 1/2+(\psi(\epsilon)-\epsilon)>1/2$$
	for all $y\in \partial B_1(v)$. This proves the claim and we end in the desired contradiction.
\end{proof}

\section{Fundamental groups of limit spaces}\label{section_fund_group}

With the help of Theorem A, we can generalize the structure results on fundamental groups of manifolds with Ricci curvature and volume lower bounds, to that of non-collapsing Ricci limit spaces. Some of the result in this section are known for the \textit{revised} fundamental groups \cite{SW1} of Ricci limit spaces.

Let $n\in\mathbb{N}$, $\kappa\in\mathbb{R}$, $D,v>0$. Let $\mathcal{M}(n,\kappa,D,v)$ be the set of all limit spaces coming from some sequence of compact $n$-manifolds $M_i$ with
$$\mathrm{Ric}_{M_i}\ge (n-1)\kappa, \quad \mathrm{diam}(M_i)\le D,\quad \mathrm{vol}(M_i)\ge v.$$	
Let $\mathcal{M}(n,\kappa,v)$ be the set of all pointed limit spaces coming from some sequence of complete $n$-manifolds $(M_i,p_i)$ with
$$\mathrm{Ric}_{M_i}\ge (n-1)\kappa,\quad \mathrm{vol}(B_1(p_i))\ge v.$$ 

\begin{cor}\label{surj}
	Let $M_i$ be a sequence of compact $n$-manifolds of
	$$\mathrm{Ric}_{M_i}\ge (n-1)\kappa, \quad \mathrm{diam}(M_i)\le D,\quad \mathrm{vol}(M_i)\ge v$$
	converging to $X\in\mathcal{M}(n,\kappa,D,v)$. Then for any $i$ large there is a surjective group homomorphism 
	$\Phi_i: \pi_1(M_i)\to \pi_1(X)$. In particular, if $M_i$ is simply connected for all $i$ large, then $X$ is simply connected as well.
\end{cor}

\begin{proof}
	The result follows from Lemma \ref{local_surj} (also see \cite[Theorem 1.1]{SW1}).
\end{proof}

\begin{rem}\label{rem_example}
We would like to point out that there is a mistake in \cite[Theorem 1.4]{SW1}. The kernel of $\Phi_i$ contains all small loops based at different points, so Anderson's result does not apply to bound the order of $\mathrm{ker}\Phi_i$. In fact, the kernel could be infinite. For example, using Ostu's construction \cite{Ot}, we have a sequence of Riemannian metrics on $(S^3 \times \mathbb RP^2) \# (S^3 \times \mathbb RP^2)$ converging to $S(S^2 \times \mathbb RP^2) \# S(S^2 \times \mathbb RP^2)$ with Ricci curvature bounded from below, where $S(S^2 \times \mathbb RP^2)$ is the spherical suspension of $S^2 \times \mathbb RP^2$. Here the  kernel is a free product $\mathbb Z_2 * \mathbb Z_2$, which is an infinite group. In fact, by taking more connected sums, the kernel may have exponential growth. 
\end{rem}

\begin{thm}
	Given $n,\kappa,D,v$, there are only finitely many isomorphic classes of fundamental groups among spaces in $\mathcal{M}(n,\kappa,D,v)$.
\end{thm}

\begin{proof}
	Anderson's original proof \cite{An1} applies through verbatim.
\end{proof}

Even though there are only finitely many isomorphic classes of fundamental groups among spaces in $\mathcal{M}(n,\kappa,D,v)$, the stability result is not true. Namely, there are spaces in $\mathcal{M}(n,\kappa,D,v)$ which are arbitrarily Gromov-Hausdorff close but have different fundamental groups (see \cite{Ot} or Remark \ref{rem_example} above). 

\begin{thm}
	Given $n,D,v$, there are positive constants $\epsilon(n,D,v)$ and $C(n,D,v)$ such that for any $X\in \mathcal{M}(n,\epsilon,D,v)$, $\pi_1(X)$ contains a normal abelian subgroup generated by at most $n$ elements of index $\le C$.
\end{thm}

\begin{proof}
	By Corollary \ref{surj}, it suffices to prove the statement for manifolds. The manifold result was proved in \cite[Theorem 1.7]{KL}. Here we present a different proof by using a result from \cite{PR}.
	
	Suppose the contrary, then we have a sequence of Riemannian manifolds $M_i^n$ with
	$$\mathrm{Ric}_{M_i}\ge -i^{-1}, \quad \mathrm{diam}(M_i)\le D,\quad \mathrm{vol}(M_i)\ge v,$$ 
	but any abelian subgroup of $\Gamma_i:=\pi_1(M_i,p_i)$ has index $\ge 2^i$. Passing to a subsequence, we obtain equivariant Gromov-Hausdorff convergence
	\begin{center}
		$\begin{CD}
		(\widetilde{M}_i,y_i,\Gamma_i) @>GH>> 
		(Y,y,G)\\
		@VV\pi_i V @VV\pi V\\
		(M_i,x_i) @>GH>> (X,x).
		\end{CD}$
	\end{center}
	By the compactness of $X$ and Cheeger-Colding's splitting theorem \cite{CC1}, $Y$ splits isometrically as $\mathbb{R}^k\times K$, where $K$ is compact. Since $M_i$ is non-collapsing, both $X$ and $Y$ has Hausdorff dimension $n$, and $G$ is a discrete group. By \cite{FY92}, there is a sequence of subgroups $H_i$ of $\Gamma_i$ such that
	$$(\widetilde{M}_i,y_i,H_i)\overset{GH}\longrightarrow (Y,y,G_0=\{e\})$$
	and $\Gamma_i/H_i$ is isomorphic to $G/G_0=G$ for all $i$ large. By \cite[Theorem 0.8]{PR} (also see Lemma 2.16 and Theorem 2.17), $H_i$ must be trivial for all $i$ large. Thus $\Gamma_i$ is isomorphic to $G$ for all $i$ large. We claim that $G$ is virtually abelian. If the claim is true, then clearly the desired contradiction follows.
	
	Let $$p:\mathrm{Isom}(\mathbb{R}^k\times K)\to \mathrm{Isom}(\mathbb{R}^k)$$
	be the natural projection. By generalized Bieberbach theorem \cite{FY92}, $p(G)$ contains a subgroup $\mathbb{Z}^r$ of finite index, where $r\le k$. Consider the exact sequence
	$$1\longrightarrow \mathrm{ker}p\cap p^{-1}(\mathbb{Z}^r) \longrightarrow p^{-1}(\mathbb{Z}^r)\overset{p}\longrightarrow \mathbb{Z}^r\longrightarrow 1.$$
	Because $\mathrm{ker}p$ is a discrete subgroup of $\mathrm{Isom}(K)$, which is compact, $\mathrm{ker}p$ is finite. By Lemma 4.4 of \cite{FY92}, $p^{-1}(\mathbb{Z}^r)$ contains an abelian subgroup generated by $r$ element and of finite index. The claim now follows from the fact that $p^{-1}(\mathbb{Z}^r)$ has finite index in $G$.
\end{proof}

\begin{thm}\label{quant_limit}
	Let $X\in\mathcal{M}(n,0,v)$. Suppose that $X$ has Euclidean volume growth of constant $\ge L$. Then $\pi_1(X)$ is finite of order $\le 1/L$. If $L>1/2$, then any loop in $B_r(x)$ is contractible in $B_{Cr}(x)$, where $x\in X$ and $r>0$.
\end{thm}

\begin{proof}
	The proof goes the same as the one of \cite{An2} and our Theorem \ref{large_vol_contra}.
\end{proof}
 
\begin{thm}
	Let $\kappa>0$ and let $X\in \mathcal{M}(n,\kappa,v)$. Then $X$ is compact and $\pi_1(X)$ is finite of order $\le C(n,\kappa,v)$.
\end{thm}

Again this was only known before assuming the universal cover is simply connected. We give a complete proof for readers' convenience.

\begin{proof}
	Let $\widetilde{X}$ be the universal cover of $X$ and let $\tilde{x}\in X$. Since relative volume comparison holds on $\widetilde{X}$, we know that
	$$\dfrac{\mathcal{H}^n(B_r(\tilde{x}))}{\mathrm{vol}(B_r^n(\kappa))}$$
	is a non-increasing function in $r$. Together with the fact that the space form of constant curvature $\kappa>0$ has diameter $\pi/\sqrt{\kappa}$, we see that
	$$\mathrm{diam}X\le \mathrm{diam}\widetilde{X}\le \pi/\sqrt{\kappa}.$$
	Then
	$$\#\pi_1(X)\le \dfrac{\mathcal{H}^n(\widetilde{X})}{\mathcal{H}^n(X)}\le\dfrac{\mathrm{vol}(B_{\pi/\sqrt{\kappa}}^n(\kappa))}{v}.$$
\end{proof} 
 
\begin{thm}
	Given $n,v>0$, there is a constant $C(n,v)$ such that the following holds. Let $X\in\mathcal{M}(n,\kappa,v)$ be a Ricci limit space and let $C_xX$ be a tangent cone of $X$ at $x\in X$. Then $C_xX$ is a metric cone $C(Z)$ with $\pi_1(Z)$ having order $\le C(n,v)$.
\end{thm}

\begin{proof}
	By \cite{CC1}, $C_xX$ is a metric cone $C(Z)$ with $\mathrm{diam}(Z)\le \pi$. $\mathcal{H}^{n-1}(Z)\ge v'$ for some $v'$ depending on $(X,x)$. Also, $\mathrm{Ric}_Z\ge n-2$ in the sense of \cite{LV,St}; in particular, relative volume comparison holds on $Z$. Thus the result follows if $Z$ is semi-locally simply connected.
	
	For a point $(1,z)\in C(Z)$, by Theorem A, $\lim_{t\to 0}\rho(t,(1,z))=0$. Choose a $T$ small so that $\rho(T,(1,z))<1/2$. For any loop $c$ in $B_T(z)\subset B_T((1,z))$, there is a homotopy $H$ contracts $c$ so that $\mathrm{im}(H)$ does not contain the vertex of $C(Z)$. We define a retraction
	$$R:C(Z)-\{\text{vertex}\} \to Z$$
	by sending $(t,z)$ to $z$. Retracting $\mathrm{im}(H)$ to $Z$ via this map $R$, we conclude that $c$ is contractible in $Z$.
\end{proof}

\Addresses
	
\end{document}